\documentclass[12pt]{amsart}
\usepackage{amsmath,amsthm,amsopn}
\usepackage{enumerate}
\usepackage[numbers]{natbib}

\def\ol{\overline}
\def\prt{\partial}
\def\wt{\widetilde}
\def\inter{\mathrm{Int}}
\def\T{{\mathcal T}}
\def\x{{\bf x}}
\def\bx{\mathbf{x}}
\def\eps{\varepsilon}
\newcommand{\X}{\mathbf X}
\newcommand{\Y}{\mathbf Y}
\newcommand{\M}{\mathcal M}
\newcommand{\XX}{\mathcal X}

\theoremstyle{plain}
\newtheorem{theorem}{Theorem}[section]
\newtheorem{lemma}[theorem]{Lemma}

\newtheorem{corollary}[theorem]{Corollary}

\theoremstyle{definition}
\newtheorem{example}[theorem]{Example}
\newtheorem{definition}[theorem]{Definition}

\theoremstyle{remark}
\newtheorem{remark}[theorem]{Remark}

\numberwithin{equation}{section}

\newcommand{\R}{\mathbb R}
\newcommand{\N}{\mathbb N}
\newcommand{\I}{{\mathbf 1}}

\newcommand{\e}{\mathrm e}
\newcommand{\set}[1]{\left\{#1\right\}}
\renewcommand{\P}[1][]{\mathbb P^{#1}}
\newcommand{\E}[1][]{\mathbb E^{#1}}
\newcommand{\F}{\mathcal F}
\DeclareMathOperator{\dist}{dist}
\DeclareMathOperator{\vol}{vol}
\DeclareMathOperator{\diam}{diam}
\DeclareMathOperator{\card}{card}
 \DeclareMathOperator{\bes}{Bes}
\newcounter{stala}
\newcommand{\stala}[1][]{\refstepcounter{stala}c_{\thestala}}
\newcommand{\Z}[1]{Z^{#1}}
\newcommand{\cadlag}{{\scshape c\`adl\`ag} }

\newcommand{\calig}{\mathcal}
\newcommand{\interior}[1]{\overset\circ{#1}}
\newcommand{\br}[1]{\left( #1\right)}
\newcommand{\bs}[1]{\left[ #1\right]}
\newcommand{\bc}[1]{\left\{ #1\right\}}
\newcommand{\ba}[1]{\left< #1\right>}
\newcommand{\bm}[1]{\left| #1\right|}
\newcommand{\bmm}[1]{\left\| #1\right\|}
\newcommand{\halfint}[1]{\left[ #1 \right)}

\newcommand{\barrier}[2]{#1\,\left|\, \vphantom{#1}#2\right.}
\newcommand{\setof}[2]{\bc{{#1}\,:\,{#2}}}
\newcommand{\restrict}[2]{\left.#1\vphantom{{#1}_{#2}}\right|_{#2}}

\begin{document}

\title{\bf Non-extinction of a Fleming-Viot particle model}
\author{
{\bf Mariusz Bieniek},  {\bf Krzysztof Burdzy}  \ and \ {\bf Sam
Finch} }

\address{Instytut Matematyki, Uniwersytet Marii Sk\l odowskiej-Curie, 20-031
Lublin, Poland}
\address{Department of Mathematics, Box 354350,
University of Washington, Seattle, WA 98195, USA}
\address{
Mathematics Institute, University of Warwick, Coventry CV4 7AL,
United Kingdom}

\email{mariusz.bieniek@poczta.umcs.lublin.pl}
\email{burdzy@math.washington.edu}
\email{S.T.J.Finch@warwick.ac.uk}

\thanks{Research supported in part by NSF Grant DMS-0600206. }

\begin{abstract}
We consider a branching particle model in which particles move
inside a Euclidean domain according to the following rules. The
particles move as independent Brownian motions until one of
them hits the boundary. This particle is killed but another
randomly chosen particle branches into two particles, to keep
the population size constant. We prove that the particle
population does not approach the boundary simultaneously in a
finite time in some Lipschitz domains. This is used to prove a
limit theorem for the empirical distribution of the particle
family.
\end{abstract}


\maketitle


\section{Introduction}\label{section:article_intro}

The paper is concerned with a branching particle system $\X_t =
(X^1_t, \dots, X^N_t)$ in which individual particles $X^j$ move
as $N$ independent Brownian motions and die when they hit the
complement of a fixed domain $D\subset \R^d$. To keep the
population size constant, whenever any particle $X^j$ dies,
another one is chosen uniformly from all particles inside $D$,
and the chosen particle branches into two particles.
Alternatively, the death/branching event can be viewed as a
jump of the $j$-th particle. See Section \ref{FVproc} for a
more detailed description of the construction.

Let $\tau_k$ be the time of the $k$-th jump of $\X_t$. Since the
distribution of the hitting time of $\prt D$ by Brownian motion
has a continuous density, only one particle can hit $\prt D$ at
time $\tau_k$, for every $k$, a.s. The construction of the process
is elementary for all $t< \tau_\infty = \lim_{k\to \infty}
\tau_k$. However, there is no obvious way to continue the process
$\X_t$ after the time $\tau_\infty$ if $\tau_\infty < \infty$.
Hence, the question of the finiteness of $\tau_\infty$ is
interesting. Theorem 1.1 in \cite{burdzymarch00} asserts that
$\tau_\infty = \infty$, a.s., for every domain $D$. Unfortunately,
the proof of that theorem contains an irreparable error (see
Example \ref{ex:sep2} below). The cited theorem might be true but
it appears to be much harder to prove that the original incorrect
argument might have suggested. Example \ref{ex:sep2} given below
shows that result cannot be generalized to arbitrary Markov
processes. L\"obus (\cite{loebus}) recently proved that
$\tau_\infty = \infty$, a.s., in Euclidean domains that satisfy
the internal ball condition. Another argument showing that
$\tau_\infty = \infty$, a.s., in domains satisfying the internal
ball condition is implicit in the proof of Theorem 1.4 of
\cite{burdzymarch00}.

In this article, we will prove that $\tau_\infty = \infty$,
a.s., if the domain $D \subset \R^d$ is Lipschitz with a
Lipschitz constant $c$ depending on $d$ and the number $N$ of
particles---see Theorem \ref{thm:exist} and Remark \ref{rem:1}
below. In addition, we prove theorems on existence and the form
of the stationary distribution of the process $\X_t$,
generalizing those in \cite{burdzymarch00}---see Section
\ref{sec:station}.

We use this attempt to rectify an error in an earlier paper to
introduce two new techniques. In the end, these techniques may
have greater interest or significance than the main theorems.
The first technique, developed in Section \ref{sec:Zproc}, is
the construction of a process of Brownian excursions in a cone,
with all excursions starting at the vertex. Such a process
exists only in cones with certain angles. The construction is
combined with a coupling argument to provide a ``lower bound''
for $\X_t$, in an appropriate sense. The process constructed
from Brownian excursions is simpler to analyze than $\X_t$.

The second technique is a new type of boundary Harnack principle (see
Section \ref{sec:bhp}). The standard boundary Harnack principle
compares two functions satisfying a PDE with the same operator, for
example, Laplacian, and different boundary conditions. Our new
version of the boundary Harnack principle compares a harmonic
function with a function $u$ satisfying $\Delta u = -1$. The reason
for proving the new form of the boundary Harnack principle is that it
allows one to compare certain probabilities and expectations, and
then use a method of proof that goes back at least to Davis
\cite{davis83}. The ``new boundary Harnack principle'' has been
proved independently by Atar, Athreya and Chen (\cite{aac}), together
with a number of other interesting theorems. We include a full proof
of the new boundary Harnack principle because it is different from
that in \cite{aac}, and ours is amenable to generalizations that will
be the subject of a forthcoming article.

Both techniques mentioned above---the Brownian excursion process and
the boundary Harnack principle---are limited to Lipschitz domains
and, moreover, the Lipschitz constant has to satisfy a certain
inequality. A natural question arises whether such special Lipschitz
domains are the largest natural family of sets where our results
hold. It turns out that they are not. In the last section of the
paper we will show that, for the two particle process, $\tau_\infty =
\infty$, a.s., in all polyhedral domains, with arbitrary angles
between the faces of the boundary. Unfortunately, our method cannot
be easily adapted to the multiparticle case, so we leave this
generalization as an open problem.

For some related results on Fleming-Viot type models in smooth
domains, see \cite{GK} and references therein. The discrete version
of the model is studied in \cite{AFG}; see also references in that
paper.

We are grateful to Zhenqing Chen, Davar Khoshnevisan and Yuval
Peres for very helpful advice.

\section{Preliminaries}

For $y=(y^1,\dotsc,y^d)\in\R^d$, let $|y|$ denote the Euclidean norm
of $y$ and let $\wt y=(y^1,\dotsc,y^{d-1})$. We will denote the open
ball with center $x$ and radius $r$ by $B(x,r)$. The closure of a set
$A$ will be denoted $\ol A$ and its interior will be denoted $\inter
A$. All constants, typically denoted by $c$ with or without
subscript, are assumed to be strictly positive and finite.

A function $F:\R^{d-1}\to\R$ is called Lipschitz if there
exists a constant $L$ such that
\begin{equation*}
  |F(x)-F(y)|\leq L|x-y|,\quad x,y\in\R^{d-1}.
\end{equation*}
Any constant $L$ satisfying the above condition will be called
a Lipschitz constant of $F$.

Consider a bounded connected open set $D\subset\R^d$, $d\geq
2$. We will call $D$ a Lipschitz domain with Lipschitz constant
$L$ if $\prt D$ can be covered by a finite number of open balls
$B_1,\dotsc,B_n$ such that for every $i=1, \dots, n$, there
exists a Lipschitz function $F_i:\R^{d-1}\to\R$ with Lipschitz
constant $L$, and an orthonormal coordinate system $CS_i$ such
that
\begin{equation*}
  D\cap B_i=\set{(y^1,\dotsc,y^d)\text{ in }CS_i:y^d>F_i(\wt y)}\cap
  B_i.
\end{equation*}

The following Harnack principles can be found in
\cite{bassPTA}.

\begin{theorem}[Harnack inequality]
  \begin{enumerate}[(a)]
    \item Suppose $0<r<R$. There exists $c=c(r,R,d)$ such
        that if $u$ is nonnegative and harmonic in
        $B(0,R)\subset \R^d$ and $x,y\in B(0,r)$, then
      \begin{equation*}
    u(x)\leq c\,u(y).
      \end{equation*}
    \item Suppose that $D\subset\R^d$ is a domain and
        $x,y\in D$ can be connected by a curve
        $\gamma\subset D$ such that
        $\inf_{z\in\gamma}\dist(z,\partial D)\geq R$. There
        exists $c=c(\gamma,R,d)$ such that if $u$ is
        nonnegative and harmonic in $D$, then
      \begin{equation*}
    u(x)\leq c\,u(y).
      \end{equation*}
  \end{enumerate}
  \label{thm:Hi}
\end{theorem}

\begin{theorem}[Boundary Harnack principle]
Suppose $D$ is a connected Lipschitz domain.  Suppose $V$ is
open, $M$ is compact and $M\subset V$. Then there exists a
constant $c=c(M,V,D)$ such that if $u$ and $v$ are two positive
and harmonic functions on $D$ that both vanish continuously on
$V\cap \partial D$, then
  \begin{equation*}
    \frac{u(x)}{v(x)}\leq c\,\frac{u(y)}{v(y)},\quad x,y\in M\cap D.
  \end{equation*}
  \label{thm:bHp}
\end{theorem}

The next theorem is a simplified version of Theorem 1 of
\cite{aikawa01}.

\begin{theorem}
Assume that $D$ is a Lipschitz domain. Then there exist constants
$r_0 = r_0(D) >0$, $c=c(D)<\infty$ and $a=a(D)>1$ such that if
$z\in\partial D$ and $0<r\leq r_0$ then for all functions $u$ and
$v$ that are bounded, positive and harmonic on $D\cap B(z,ar)$,
and vanishing continuously on $\partial D\cap B(z,ar)$, we have
  \begin{equation*}
    \frac{u(x)}{v(x)}\leq c\frac{u(y)}{v(y)},\quad x,y\in D\cap B(z,r).
  \end{equation*}
  \label{thm:unibHp}
\end{theorem}

\begin{remark}\label{rem:a29-1}
Theorem \ref{thm:unibHp} can be used to estimate the constant
$c(M,V,D)$ in Theorem \ref{thm:bHp} as follows. Suppose that
$r_0$ and $a$ are as in Theorem \ref{thm:unibHp} and we can
find balls $B_i(x_i, r_i)$, $i=1, \dots, n$, and $B'_j(y_j,
\rho)$, $j=1,\dots,m$, $\rho>0$, $r_i \leq r_0$, $x_i \in \prt
D$, $y_j \in D$, $M \subset \bigcup_i B_i(x_i, r_i) \cup
\bigcup_j B'_j(y_j, \rho)$, and $ \bigcup_i B_i(x_i, a r_i)
\subset V $ and $  \bigcup_j B'_j(y_j, 2\rho) \subset D$. A
simple chaining argument based on Theorems \ref{thm:Hi} and
\ref{thm:unibHp} then shows that the constant $c(M,V,D)$ in
Theorem \ref{thm:bHp} depends only on $n,m$ and $D$.
\end{remark}

Next we recall some notation and results from
\cite{burkholder77}. Fix $d\geq 2$ and $p>0$. Let
\begin{equation}\label{eq:h}
  h(\theta)=h_{p,d}(\theta)=F\left( -p,p+d-2;(d-1)/2;(1-\cos\theta)/2 \right),
\end{equation}
where
\begin{equation*}
  F(a,b;c;x)=\sum_{k=0}^{\infty}\frac{(a)_k(b)_k}{(c)_kk!}x^k,\quad |x|<1,
\end{equation*}
denotes the hypergeometric function and $(a)_k=a(a+1)\dotsc(a+k-1)$, $(a)_0=1$. The
function $h$ has at least one zero in $(0,\pi)$; let $\theta_{p,d}$ denote the smallest
one. The quantity $\theta_{p,d}$ is strictly decreasing in $p$ for any fixed $d\geq 2$,
and strictly increasing to $\pi/2$ in $d$ for any fixed $p>1$. In particular, if $p=2$,
then
\begin{equation*}
  h_{2,d}(\theta)=1-\frac{d}{d-1}\sin^2\theta,
\end{equation*}
$\theta_{2,d}=\arccos\frac{1}{\sqrt d}$ and
$\cot\theta_{2,d}=\frac{1}{\sqrt{d-1}}$. Therefore
$\theta_{2,2}=\pi/4$ and $p<2$ is equivalent to
$\cot\theta_{p,d}<\frac{1}{\sqrt{d-1}}$.

For $d\geq 2$ and $p>0$ we let $\theta$ be the angle between
$y$ and $(0,\dotsc,0,1)$,
\begin{equation*}
  K_{p,d}=\set{y\in\R^d:y\neq 0,\, 0\leq\theta<\theta_{p,d}},
\end{equation*}
and let $O$ denote the axis of $K_{p,d}$. Obviously $p<p'$
implies $K_{p',d}\subset K_{p,d}$. We will drop the subscripts
$p$ and $d$ and write $K$ instead of $K_{p,d}$ whenever there
is no danger of confusion.

The function $ v(x)=|x|^p h(\theta)$, where $h$ is given by
\eqref{eq:h}, is positive and harmonic inside $K$ and
continuous on $\ol K$ with $v(x)=0$ for $x\in\partial K$.

Let $(\P[x],X_t)$ be $d$-dimensional Brownian motion and for a
Borel set $A\subset \R^d$ define
\begin{equation}\label{eq:hittime}
  T_A=\inf\left\{ t>0:X_t\in A \right\}.
\end{equation}

\begin{lemma}
Let $F$ denote the intersection of $K=K_{p,d}$ and a hyperplane
orthogonal to $O$. Let $z_0$ be the point of intersection of
$O$ with $F$ and assume that $z_0 \in K$. There exists
$c=c(p,d)$ such that for all $z_1,z_2\in O$ with
$|z_0|<|z_1|<|z_2|$, we have
\begin{equation}
  \label{eq:mainestimate}
  \frac{\P[z_2]\left( T_F<T_{\partial K} \right)}{\P[z_1]\left( T_F<T_{\partial
  K} \right)} \geq c\,\left( \frac{|z_2|}{|z_1|}  \right)^{2-d-p}.
\end{equation}
\label{lem:main}
\end{lemma}

\begin{proof}
Let $K_*$ be the unbounded component of $K\setminus F$ and
\begin{equation*}
  u(z)=\P[z]\left( T_F<T_{\partial K} \right),\quad z\in K_*.
\end{equation*}
Then $u$ is positive and harmonic in $K_*$ and continuous on
$\ol K_* \setminus (F\cap \prt K)$, with $u(z)=0$ for
$z\in\partial K \setminus F$. It is easy to see that $u(x) \to
0$ as $|x| \to \infty$.

If $I(x)=x/|x|^2$, then the function $\wt u(x)=|x|^{2-d}
u\left( I(x) \right) $ is positive and harmonic in $\wt K =
I(K_*)$ (see Lemma 1.18 of \cite{bassPTA}). The function $\wt
u$ vanishes continuously on $\prt \wt K \setminus I(F)$. Let
$K' = (1/2) \wt K$. Recall that $ v(x)=|x|^p h(\theta) $ is
positive and harmonic inside $K$ and continuous on $\ol K$ with
$v(x)=0$ for $x\in\partial K$. By the boundary Harnack
principle,
\begin{equation}\label{eq:1a}
  \frac{\wt u(z)}{\wt u(z')}\geq c\frac{v(z)}{v(z')},
\end{equation}
for $z,z'\in K'$, where $c$ depends on $\wt K$ and $K'$ and
does not depend on $z$ and $z'$. Note that $ u(x)=|x|^{2-d} \wt
u\left( I(x) \right) $. Hence, for $z_1,z_2\in O \cap I(K')$,
\begin{equation*}
  \frac{u(z_2)}{u(z_1)}
 = \frac{ |z_2|^{2-d} \wt u(I(z_2))}
 {|z_1|^{2-d}\wt u(I(z_1))}
 \geq c\frac{ |z_2|^{2-d} v(I(z_2))}
 {|z_1|^{2-d} v(I(z_1))}
 = c\frac{ |z_2|^{2-d} |z_2|^{-p} h(0)}
 {|z_1|^{2-d}  |z_1|^{-p} h(0)}
  = c\left( \frac{|z_2|}{|z_1|} \right)^{2-p-d}.
\end{equation*}
The inequality holds for all $z_1,z_2\in O \cap I(K_*)$
(possibly with a different value of $c$) because the function
$u$ is bounded below and above on $O \setminus I(K')$ by
strictly positive and finite constants. This completes the
proof of \eqref{eq:mainestimate}.
\end{proof}

We will use the following estimate in the proof of Lemma
\ref{lem:4}.

\begin{lemma}\label{lem:5a}
There exists a cone $K'\subset K=K_{p,d}$ and a constant
$c=c(K,K')$ such that for $x\in K'$ and $t\geq|x|^2$,
  \begin{equation}\label{eq:R>t}
    c^{-1}\left( \frac{t}{|x|^2} \right)^{-\frac{p}{2}}
    \leq \P[x]\left( T_{\partial K}>t \right)
    \leq c\left( \frac{t}{|x|^2} \right)^{-\frac{p}{2}}.
  \end{equation}
\end{lemma}

\begin{proof}
See \cite{banuelossmits97,bassburdzy96,deblassie87} or
\cite{meyrewerner95}.
\end{proof}

\section{A boundary Harnack principle}
\label{sec:bhp}

Let $D\subset\R^d$, $d\geq 2$, be a bounded Lipschitz domain
and let $A\subset D$ be a compact set with $\inter
A\neq\emptyset$. For $x\in D$, define
\begin{align*}
  f(x)&=\P[x](T_A<T_{\partial D}),\\
  g(x)&=\E[x]T_{\partial D}.
\end{align*}

\begin{theorem}\label{thm:main}
Assume that the Lipschitz constant $L$ of $D$ satisfies $L<
\frac{1}{\sqrt{d-1}}$. Then there exists a constant $c=c(A,D)$
such that for all $x\in D$,
  \begin{equation}\label{eq:main}
    \frac{1}{c}\leq\frac{f(x)}{g(x)}\leq c.
  \end{equation}
\end{theorem}

\begin{remark}\label{rem:kb1}
The condition $L< \frac{1}{\sqrt{d-1}}$ is sharp. See Example
\ref{ex:kb2} below.
\end{remark}

\begin{proof}[Proof of RHS of \eqref{eq:main}]
Since $A$ is compact, $\inf_{x\in A} \dist(x, D^c) =c_1
>0$. Therefore,
  \begin{equation*}
    \inf_{x\in A}\E[x]T_{\partial D}
 \geq \inf_{x\in A}\E[x]T_{\partial B(x,c_1)}
    =c_2>0.
  \end{equation*}
By the strong Markov property applied at $T_A$, we have for
$x\in D$,
  \begin{equation*}
    \E[x]T_{\partial D}\geq c_2\P[x](T_A<T_{\partial D}),
  \end{equation*}
  which implies the RHS of \eqref{eq:main}.
\end{proof}

\begin{proof}[Proof of LHS of \eqref{eq:main}]
Since $D$ is a bounded Lipschitz domain with Lipschitz constant
$L< \frac{1}{\sqrt{d-1}}$, it is easy to see that there exist
$p\in(0,2)$ and $\rho>0$ with the following properties.

(i) $\dist(A, \prt D) > 2\rho$.

(ii) Consider any $x\in D$ with $\dist(x,\prt D) < \rho
2^{-5}$. Then there exists $x_0\in \prt D$ and an orthonormal
coordinate system $CS=CS_{x_0}$ with the following properties.
The origin of $CS$ is $x_0$, $K_{p,d} \cap B(x_0, 2 \rho)
\subset D \cap B(x_0, 2 \rho)$, and $x\in O$ (that is, $x$
belongs to the axis of $K_{p,d}$). For $r>0$ and integer $k$,
let
  \begin{align*}
E^*_r &=\left\{ y\in\R^d \text {  in  } CS:
|\wt y-\wt x_0|\leq r\tan(\theta_{p,d}),
    |y^d-x_0^d|\leq r \right\},\\
\wt E_k & = E^*_{2^{-k}}.
  \end{align*}
We can choose $x_0$ and $CS$ so that for some Lipschitz
function $F=F_{x_0}:\R^{d-1}\to\R$ with Lipschitz constant $L$,
and all $k$ such that $2^{-k} \leq \rho$,
\begin{equation*}
  D\cap \wt E_k=\set{(y^1,\dotsc,y^d)\text{ in }CS :y^d>F(\wt y)}\cap
  \wt E_k.
\end{equation*}

We fix $x\in D$ with $\dist(x,\prt D) < \rho 2^{-5}$ and the
corresponding coordinate system $CS$ for the rest of the proof.

Let $E_k = \wt E_k \setminus \wt E_{k+1}$ and $C_k=\inter(D\cap
E_k)$ for $k=N_0,\dotsc,N_1$, where
  \begin{align*}
  N_0 = \min\{k: 2^{-k} \leq \rho\}, \qquad
    N_1=\max\left\{k: |x| = x^d \leq 2^{-k-3} \right\}.
  \end{align*}
Also let $C_{N_0-1} = \inter\left( D \setminus \wt
E_{N_0}\right)$ and $C_{N_1+1} = \inter(D\cap \wt E_{N_1+1})$.

Note that $C_i\cap C_j=\emptyset$ if $i\neq j$, and $D= \ol
C_{N_0-1}\cup\dotsc \cup \ol C_{N_1+1}$.

  Let $G(x,y)$ denote the Green function for Brownian motion killed on exiting $D$. Then
  \begin{equation}\label{eq:gx}
    g(x)=\E[x]T_{\partial D}=\int_D G(x,y)\,dy
    =\sum_{k=N_0-1}^{N_1+1}\int_{C_k}G(x,y)\,dy.
  \end{equation}

For $k=N_0,\dotsc,N_1$ denote by $y_k$ the midpoint of the line
segment being the intersection of $C_k$ with $x^d$-axis in
$CS$. In other words, $\{y_k\} = \prt E^*_{(3/4)2^{-k}} \cap
O$. Fix $k$ and $j$ such that $j\geq 1$, $k\geq N_0$, $j+k\leq
N_1$ and consider the points $y_k$ and $y_{k+j}$.

Let
  \begin{equation*}
    F_k=\ol C_k\cap\ol C_{k+1}\cap K_{p,d},
  \end{equation*}
and
 \begin{equation*}
   u(z)=\P[z](T_{F_{k+j}}<T_{\partial K_{p,d}}).
 \end{equation*}
By Lemma \ref{lem:main},
  \begin{equation*}
    u(y_k)\geq c_1u(y_{k+j})\left( \frac{2^{-k}}{2^{-k-j}} \right)^{2-p-d}
    = c_1u(y_{k+j})2^{j(2-p-d)},
  \end{equation*}
where $c_1=c_1(p,d)$. By scaling properties of Brownian motion,
$u(y_{k+j})=c_2=c_2(p,d)$, that is, $u(y_{k+j})$ depends only
on $p$ and $d$. We obtain
  \begin{equation}\label{eq:mainineq}
    \P[z](T_{F_{k+j}}<T_{\partial K_{p,d}})\geq c_3 2^{-j(p+d-2)},
  \end{equation}
  where $c_3=c_3(p,d)$.

Let
  \begin{equation*}
    v(z)=\P[z]\left( T_{F_{k+j}}<T_{\partial D} \right).
  \end{equation*}
Note that $v(y_{k+j})\leq 1$ and $v(y_k)\geq u(y_k)\geq c_3
2^{-j(p+d-2)}$, by \eqref{eq:mainineq}.

We will apply Theorem \ref{thm:bHp} with $M = \prt
E^*_{(3/4)2^{-k-j}}$ and $V = E_{k+j}$. It follows from Remark
\ref{rem:a29-1} that the constant $c_5=c(M,V,D)$ may be chosen
independent of $k$ and $j$. The boundary Harnack principle
implies that
  \begin{equation}\label{eq:a29-2}
    \frac{G(x,z)}{G(x,y_{k+j})}\geq c_5 \frac{v(z)}{v(y_{k+j})},
  \end{equation}
for $z\in D\cap M$. The harmonic functions $G(x,\,\cdot\,)$ and
$v$ have zero boundary values on $\prt D \setminus \ol
E^*_{(3/4)2^{-k-j}}$, so the inequality \eqref{eq:a29-2}
extends to all $z \in D \setminus E^*_{(3/4)2^{-k-j}}$, in
particular, it applies to $z= y_k$. Hence,
  \begin{equation}\label{eq:1}
    \frac{G(x,y_k)}{G(x,y_{k+j})}\geq c_5 \frac{v(y_k)}{v(y_{k+j})}
    \geq c_5 c_3 2^{-j(p+d-2)} =c_6  2^{-j(p+d-2)}.
  \end{equation}

Now consider the function
  \begin{equation*}
    h_m(z)=\P[z]\left( T_{\wt E_{m+2}}<T_{\partial D} \right).
  \end{equation*}
By the scaling properties of Brownian motion, $h_m(y_m)\geq c_7>0$
for all $m=N_0,\dotsc, N_1$. By the boundary Harnack principle
(Theorem \ref{thm:bHp}) applied to $u(z)=G(x,z)$, $v(z)=h_m(z)$,
$M=\ol C_m$ and $V=\inter( \wt E_{m-1} \setminus
E^*_{(3/4)2^{-m-1}})$, we have
  \begin{equation*}
    \frac{G(x,y)}{h_m(y)}\leq c_8\frac{G(x,y_m)}{h_m(y_m)}
  \end{equation*}
for $y\in C_m$, where $c_8$ depends only on $D$, by Remark
\ref{rem:a29-1}. Therefore, for $y\in C_m$,
  \begin{equation}\label{eq:kb7}
    G(x,y)\leq c_8\,G(x,y_m)\frac{h_m(y)}{h_m(y_m)}
    \leq c_8\,\frac{1}{c_7}\,G(x,y_m)
    =c_9\,G(x,y_m).
  \end{equation}
This implies
  \begin{equation}\label{eq:2}
      \int_{C_{k+j}}G(x,y)\,dy\leq c_9G(x,y_{k+j})\vol(C_{k+j})
      \leq c_{10}\,2^{-d(k+j)}G(x,y_{k+j}),
  \end{equation}
where $c_{10}$ depends only on $D$.

On the other hand, by the usual Harnack inequality,
  \begin{equation*}
    G(x,y)\geq c_{11} G(x,y_k)
  \end{equation*}
for $y\in B_k=B(y_k,2^{-k-1})$, because $B(y_k,2^{-k-1})
\subset D \setminus \{x\}$. This implies that
  \begin{equation}\label{eq:3}
      \int_{C_k}G(x,y)\,dy\geq c_{11}G(x,y_k)\vol(B_k)
      =c_{12} 2^{-kd} G(x,y_k),
  \end{equation}
where $c_{12}$  does not depend on $k$.

Combining \eqref{eq:1}, \eqref{eq:2} and \eqref{eq:3} we have
  \begin{equation*}
    \int_{C_{k+j}}G(x,y)\,dy\leq c_{13}2^{j(p-2)}\int_{C_k}G(x,y)\,dy,
  \end{equation*}
where $c_{13}=c_{13}(D)$. Fix  $q< 1$. Since $p\in(0,2)$, we may
choose $j$ so large that $c_{13}2^{j(p-2)}\leq q<1$. Let
$a_k=\int_{C_{k}}G(x,y)\,dy$, then
  \begin{equation*}\label{eq:seq}
    a_{k+j}\leq q a_k, \quad k=N_0,\dotsc,N_1-j.
  \end{equation*}
Let $N_2 = \min(N_1, N_0+j-1)$. The last inequality implies
that
 \begin{equation}\label{eq:ET<EsB}
    \sum_{k=N_0}^{N_1}a_k
= \sum_{k=N_0}^{N_2} \sum_{m=0}^\infty a_{k+mj} \I_{\{k+mj \leq N_1\}}
\leq \sum_{k=N_0}^{N_2} \sum_{m=0}^\infty a_{k} q^m
= c_{14} \sum_{k=N_0}^{N_2} a_k.
  \end{equation}

Recall that $G(x,\,\cdot\,)$ has zero boundary values on $\prt
D$, so it is bounded by $\sup_{z\in C_{N_0}} G(x,z)$ on the set
$D\setminus \wt E_{N_0}$. This and \eqref{eq:kb7} imply that
$\sup_{z\in D\setminus \wt E_{N_0}} G(x,z) \leq c_{15} G(x,
y_{N_0})$. We use \eqref{eq:3} to see that
\begin{align}\label{eq:kb8}
a_{N_0-1} &=
\int_{C_{N_0-1}}G(x,y)\,dy
\leq c_{15} G(x, y_{N_0}) \vol(C_{N_0-1})\\
&\leq c_{15} \vol(C_{N_0-1})
c_{12}^{-1} 2^{N_0 d} \int_{C_{N_0}}G(x,y)\,dy
= c_{16} a_{N_0}.\nonumber
\end{align}
Recall the definition of $N_0$ to see that $c_{16}$ depends
only on $D$.

The following calculation is presented in the case $d\geq 3$
only. The case $d=2$ requires minor modifications and is left
to the reader.

Let $\wt G(x,y)$ denote the Green function for Brownian motion
in $\R^d$, and let $\ol G(x,y) $ be the Green function for
Brownian motion in $B(x,2^{-N_1-4})$. It is well known that for
$d\geq 3$, $\wt G(x,y)=c_{17}|x-y|^{2-d}$, where $c_{17}$
depends on $d$, and $\ol G(x,y) = \wt G(x,y) - \wt G(x, z)$,
for $y\in B(x,2^{-N_1-4})$ and $z\in \prt B(x,2^{-N_1-4})$. It
follows that for $|y-x| \leq 2^{-N_1-5}$,
\begin{equation}\label{eq:sep6-3}
 \ol G(x,y)\geq c_{18} \wt G(x,y).
\end{equation}
We have $G(x,y)\leq\wt G(x,y) $ for $y\in D$, and
$\int_{B(x,r)}\wt G(x,y)dy=c_{19} r^2$. Therefore,
\begin{equation}\label{eq:sep6-1}
  a_{N_1+1}=\int_{C_{N_1+1}}G(x,y)dy\leq\int_{C_{N_1+1}}\wt G(x,y)dy\leq
  \int_{B(x,\diam(\wt E_{N_1+1}))}\wt G(x,y)dy=c_{20} 2^{-2N_1}.
\end{equation}
Since $ B(x,2^{-N_1-4}) \subset D$,
\begin{equation}\label{eq:sep6-2}
  G(x,y)\geq \ol G(x,y).
\end{equation}
Put $y_{N_1+1}=(\wt x,x_d+2^{-N_1-5})$. Then by
\eqref{eq:sep6-3} and \eqref{eq:sep6-2},
\begin{equation*}
  G(x,y_{N_1+1})\geq c_{18}\wt G(x,y)=c_{21}(2^{-N_1})^{2-d}.
\end{equation*}
Moreover, by the usual Harnack inequality,
\begin{equation*}
  G(x,y)\geq c_{22} G(x,y_{N_1+1}),
\end{equation*}
for $y\in B(y_{N_1},2^{-N_1-2})$. Therefore,
\begin{equation}\label{eq:sep6-4}
  \begin{split}
    a_{N_1}&=\int_{C_{N_1}}G(x,y)dy\geq\int_{B(y_{N_1},2^{-N_1-2})}G(x,y)dy\\
    &\geq c_{22} G(x,y_{N_1+1})\vol(B(y_{N_1},2^{-N_1-2}))
    \geq c_{23} (2^{-N_1})^{2-d}\cdot
    2^{-N_1d}=c_{24} 2^{-2N_1}.
  \end{split}
\end{equation}
Combining \eqref{eq:sep6-1} and \eqref{eq:sep6-4}, we obtain
\begin{equation}  \label{eq:a_N}
  a_{N_1+1}\leq c_{25} a_{N_1}.
\end{equation}

Let $C_*=C_{N_0-1}\cup \dotsc\cup C_{N_2}$ and note that
$A\subset C_*$. Let $\sigma_{C_*}=\int_0^{T_{\partial
D}}1_{\set{X_s\in C_*}}\,ds$. Then \eqref{eq:ET<EsB},
\eqref{eq:kb8} and \eqref{eq:a_N} imply that
  \begin{equation}\label{eq:ETD<EsB}
    \E[x]T_{\partial D}\leq c_{26} \E[x]\sigma_{C_*}.
  \end{equation}
Since $D$ is bounded, $ \sup_{z\in D}\E[z]T_{\partial D}=c_{27}
<\infty$. By the strong Markov property applied at the hitting
time of $C_*$, for $z\in D$,
  \begin{equation*}\label{eq:EsB<PTB}
    \E[z]\sigma_{C_*}\leq c_{27}\,\P[z](T_{C_*}<T_{\partial D}).
  \end{equation*}
This and \eqref{eq:ETD<EsB} yield
  \begin{equation}\label{eq:kb9}
    \E[x]T_{\partial D}\leq c_{28} \P[x](T_{C_*}<T_{\partial D}).
  \end{equation}

Consider functions
  \begin{eqnarray*}
    \xi_1(z)&=& \P[z](T_A<T_{\partial D}),\\
    \xi_2(z)&=& \P[z](T_{C_*}<T_{\partial D}).
  \end{eqnarray*}
Both functions are positive and harmonic in $D\setminus \ol
C_*$, and continuous on $\ol D\setminus\ol C_*$ with
$u(z)=v(z)=0$ for $z\in\partial D\setminus\ol C_*$. We apply
the boundary Harnack principle with $V = D\setminus \ol C_*$
and $M = \wt E_{N_2+1}$ to see that
  \begin{equation}\label{eq:a29-7}
    \frac{\xi_1(x)}{\xi_2(x)}
    \geq c_{29}\frac{\xi_1(y_{N_2+1})}{\xi_2(y_{N_2+1})}.
  \end{equation}
We use Remark \ref{rem:a29-1} to see that $c_{29}$ may be
chosen so that it depends only on $D$. It follows from the
definitions of $N_0, N_2$ and $j$ that for some constant
$c_{30}$, we have $\dist(y_{N_2+1}, \prt D) > c_{30}$. This
implies that $\xi_1(y_{N_2+1}) = \P[y_{N_2+1}](T_A<T_{\partial
D}) \geq c_{31}$, for some $c_{31}$ depending only on $D$. We
obtain from \eqref{eq:a29-7} that $\xi_1(x)/\xi_2(x) \geq
c_{29}c_{31}$, and this combined with \eqref{eq:kb9} gives
  \begin{equation*}
    \E[x]T_{\partial D}\leq (c_{28}/c_{29}c_{31}) \P[z](T_A<T_{\partial D}).
  \end{equation*}
We have proved the LHS of \eqref{eq:main} for $x$ satisfying
$\dist(x,\prt D) \leq \rho 2^{-5}$.

It is easy to check that $\inf\{f(x) : \dist(x,\prt D) \geq
\rho 2^{-5}\} >0$ and $\sup\{g(x) : x\in D\} < \infty$, so the
LHS of \eqref{eq:main} holds for all $x\in D$.
\end{proof}

\begin{example}\label{ex:kb2}
The condition $L< \frac{1}{\sqrt{d-1}}$ in Theorem
\ref{thm:main} is sharp. To see this, note that for any $L>
\frac{1}{\sqrt{d-1}}$ there is a $p>2$, such that the cone
$K=K_{p,d}$ is a Lipschitz domain with the Lipschitz constant
$L$. Let $r>0$ be such that for every $x\in O$, $B(x, r|x|)
\subset K$. Then $g(x) = \E[x] T_{\prt K} \geq \E[x] T_{\prt
B(x, r|x|)} \geq c_1 r^2 |x|^2$. Recall that $f(x) = \P[x] (T_A
< T_{\prt K})$ and let $u(x) = |x|^p h_{p,d}(\theta)$. By the
boundary Harnack principle applied to $f$ and $u$ in a
neighborhood of 0, $f(x) \leq c_2 |x|^p$ for $x\in O$, $|x|<1$.
Since $p>2$, we cannot have $f(x) \geq c_3 g(x)$ in a
neighborhood of 0. The domain $K$ is unbounded but it is easy
to extend the argument to $K \cap B(0,1)$.
\end{example}

\section{Construction of an auxiliary process from Brownian
excursions}\label{sec:Zproc}

Let $\Omega$ denote the family of all functions
$\omega:[0,\infty)\to \R^d\cup\set{\delta}$ continuous up to
their lifetime $R(\omega)=\inf\set{t\geq 0:\omega(t)=\delta}$
and constantly equal to $\delta$ for $t\geq R$, where $\delta$
denotes the coffin state outside $\R^d$.  Let $X$ be the
canonical process on $\Omega$, i.e., $X_t(\omega)=\omega(t)$
and let $\P[x]$ denote the distribution of Brownian motion
starting from $x\in\R^d$. As in \eqref{eq:hittime}, for a Borel
set $A\subset\R^d$ let $T_A=\inf\set{t>0:X_t\in A}$. Let
$K=K_{p,d}$ for some $p>0$, and let $X'$ denote the process
\begin{equation*}
  X'_t=
  \begin{cases}
    X_t,&\text{for $t<T_{\partial K}$,}\\
    \delta,&\text{otherwise,}
  \end{cases}
\end{equation*}
i.e., $X'$ is the process $X$ killed on exiting $K$. If $X$ has
the distribution $\P[x]$, then $X'$ is called Brownian motion
in $K$ and its distribution is denoted by $\P[x]_K$.

Let $U$ denote the family of all functions
$\omega:[0,\infty)\to K\cup\set{\delta}$ such that
$\omega(0)=0$, continuous up to their lifetime $R$. Let $H^0$
denote a standard excursion law of Brownian motion in $K_{p,d}$
starting from $0$. Namely, $H^0$ is a nonnegative and
$\sigma$-finite measure on $\Omega$ such that $X$ is strong
Markov under $H^0$ with the $\P_K$ transition probabilities and
$H^0\left( \lim_{t\to 0} X_t \neq 0\right)=0$. We have
$H^0(\Omega\setminus U)=0$. The existence of $H^0$ follows from
results of \cite{maison75} and \cite{burdzy87}.

\begin{lemma} \label{lem:4}
  There exists $c\in(0,\infty)$ such that
  \begin{equation}\label{eq:rozklad}
    H^0(R>t)=c\, t^{-\frac{p}{2}},\quad t>0.
  \end{equation}
\end{lemma}

\begin{proof}
Let $y_\varepsilon=(0,\dotsc,0,\varepsilon)\in\R^d$ and let
$G_K(x,y)$ denote the Green function for $K$. By Theorem 4.1 of
\cite{burdzy87},
  \begin{equation}\label{eq:eqlimit1}
    H^0\left( R>t \right)=c_1\lim_{\substack{z\to 0\\z\in K}}\frac{\P[z](T_{\partial
    K}>t)}{G_{K}(z,y_1)}.
  \end{equation}
By Theorem 2.2 of \cite{burdzy87}, which is an improvement of
the boundary Harnack principle, there exists $c(K,\varepsilon)$
such that for all functions $h_1$ and $h_2$ which are positive
and harmonic in $K$ and vanish continuously on $\partial K$, we
have
  \begin{equation*}
    c(K,\varepsilon)^{-1}\frac{h_1(y)}{h_2(y)}
    \leq \frac{h_1(x)}{h_2(x)}
    \leq c(K,\varepsilon)\frac{h_1(y)}{h_2(y)},
  \end{equation*}
for all $x,y\in K\cap B(0,\varepsilon)$, and
$\lim_{\varepsilon\to 0} c(K,\varepsilon)=1$. Therefore, the
limit
  \begin{equation*}
    \lim_{\substack{z\to 0\\z\in K}}\frac{h_1(z)}{h_2(z)}
  \end{equation*}
exists and belongs to $(0,\infty)$ for all functions $h_1,h_2$
satisfying the above assumptions. We apply this claim to
$h_1(z)=G_K(z,y_1)$ and $h_2(z)=|z|^p h(\theta)$, to conclude
that
  \begin{equation*}
    \lim_{\varepsilon\to 0}\frac{G_K(y_\varepsilon,y_1)}{\varepsilon^p}=c\in(0,\infty),
  \end{equation*}
  and
  \begin{equation}\label{eq:eqlimit2}
    H^0(R>t)=c\lim_{\varepsilon\to 0}\frac{\P[y_\varepsilon](T_{\partial
    K}>t)}{\varepsilon^p}.
  \end{equation}
By Lemma \ref{lem:5a},
  \begin{equation*}
    c^{-1}t^{-\frac{p}{2}}\leq \frac{\P[y_\varepsilon](T_{\partial
    K}>t)}{\varepsilon^p}\leq ct^{-\frac{p}{2}},
  \end{equation*}
for $t\geq \varepsilon^2$ which implies
$c^{-1}t^{-\frac{p}{2}}\leq H^0(R>t)\leq ct^{-\frac{p}{2}}$,
for $t\geq 0$. Therefore $H^0(R>1)$ is a positive and finite
number.

Now mimicking the proof of Proposition 5.1 of \cite{burdzy87},
using \eqref{eq:eqlimit2} instead of \eqref{eq:eqlimit1}, we
easily see that if $\set{X(t),\,t\geq 0}$ has the distribution
$H^0$, then for every $a>0$ the scaled process
$\set{\sqrt{a}X(t/a),\,t\geq 0}$ has the distribution
$a^{p/2}H^0$. In particular, for every $a>0$
  \begin{equation*}
    H^0(R>t)=a^{p/2}H^0(R>at), \quad t\geq 0,
  \end{equation*}
and putting $a=1/t$ we obtain \eqref{eq:rozklad} with
$c=H^0(R>1)$.
\end{proof}

Let $\lambda$ denote the Lebesgue measure on $\R_+=[0,\infty)$
and let $\mathcal P$ be a Poisson point process on $\R_+\times
U$ with characteristic measure $\lambda\times H^0$, i.e.,
$\mathcal P$ is a random subset of $\R_+\times U$ such that for
every pair $A_1,A_2$ of disjoint nonrandom subsets of
$\R_+\times U$, $\card(\mathcal P\cap A_1)$ and $\card(\mathcal
P\cap A_2)$ are independent random variables with Poisson
distributions with means $(\lambda\times H^0)( A_1)$ and
$(\lambda\times H^0)( A_2)$, respectively (\cite{ito72}). With
probability 1, there are no two points with the same first
coordinate, and therefore the elements of $\mathcal P$ may be
unambiguously denoted by $(t,e_t)$. Let
\begin{equation*}
  R_t=\inf\set{s>0:e_t(s)=\delta}.
\end{equation*}
By abuse of notation, for a generic element $e$ of $U$ we will
write
\begin{equation*}
  R(e)=\inf\set{s>0:e(s)=\delta}.
\end{equation*}

\begin{lemma}
  \label{lem:5}
  If $p\in (0,2)$, then for every $s>0$,
  \begin{equation*}
    \sum_{t\leq s}R_t<\infty,\quad {\rm a.s.}
  \end{equation*}
\end{lemma}

\begin{proof}
We use Theorem 4.6 of \cite{ito72}: if $\varphi:\R_+\times
U\to\R_+$ is a measurable function, then
  \begin{equation*}
    \sum_{t}\varphi(t,e_t)<\infty,\quad {\rm a.s.}
  \end{equation*}
  iff
  \begin{equation*}
    \iint_{\R_+\times U}(\varphi(t,e)\wedge 1) dt H^0(de)<\infty.
  \end{equation*}
  In particular, if $\varphi(t,e)=R(e)1_{[0,s]}(t)$, then
  \begin{equation*}
    \sum_{t\leq s}R_t<\infty,\quad {\rm a.s.}
  \end{equation*}
  iff
  \begin{equation*}
    \iint_{[0,s]\times U} (R(e)\wedge 1) dt H^0(de)<\infty.
  \end{equation*}
If we let $U^-=\set{e\in U:R(e)\leq 1}$ and $U^+=\set{e\in
U:R(e)>1}$ then
  \begin{equation*}
    \iint_{[0,s]\times U} (R(e)\wedge 1) dt H^0(de)
    =s\int_{U^-}R(e)H^0(de)+sH^0(U^+).
  \end{equation*}
By Lemma \ref{lem:4},
  \begin{equation*}
    H^0(U^+)=\int_1^\infty H^0(R\in dt)=c\int_1^\infty t^{-p/2-1}dt<\infty,
  \end{equation*}
because $p>0$, and
  \begin{equation*}
    \int_{U^-}RdH^0=\int_0^1tH^0(R\in dt)=c\int_0^1 t\cdot t^{-p/2-1}dt<\infty,
  \end{equation*}
because $p<2$.
\end{proof}

By Lemma \ref{lem:5}, the following process $Z$ with values in
$K_{p,d}\cup\set{0}$, where $p\in(0,2)$, is well defined. For
every $t>0$ the formula
\begin{equation*}
  t=
  \begin{cases}
    s+\sum_{u<r}R_u,&\text{if there exists $r<t$ such that
    $\sum_{u<r}R_u < t \leq \sum_{u\leq r}R_u$},\\
    \sum_{u<r}R_u,&\text{otherwise},
  \end{cases}
\end{equation*}
defines a unique pair $(r,s)$ with $r>0$ and $s\in[0,R_r)$ (in
the first case). Then we define
\begin{equation}\label{eq:a31-1}
  Z_t=
  \begin{cases}
    e_r(t),&\text{if there exists $r<t$ such that
    $\sum_{u<r}R_u < t < \sum_{u\leq r}R_u$},\\
    0,&\text{otherwise}.
  \end{cases}
\end{equation}
The process $Z$ takes values in $K\cup\set{0}$. Let $\sigma_t =
\sum_{s\leq t}R_t$ for $t\geq 0$.

\begin{lemma}\label{lem:stable}
The process $\sigma$ is a stable subordinator with index $p/2$.
\end{lemma}

\begin{proof}
The process $\sigma$ is increasing and has values in
$[0,\infty)$. Its paths are right-continuous with left limits.
Note that $\{(t, R(e_t))\}_{e\in \mathcal P}$ is a Poisson
point process on $\R_+\times\R_+$ with characteristic measure
$\lambda\times\Pi$, where $\Pi$ is given by
  \begin{equation*}
    \Pi(dx)=H^0(R\in dx)=c\,x^{-p/2-1}dx,
  \end{equation*}
where the last formula follows from Lemma \ref{lem:4}. This
implies that $\sigma$ is a process with independent and
stationary increments, so $\sigma$ is a L\'evy process.
Moreover $\sigma$ is a subordinator, since it has values in
$[0,\infty)$ only. We use calculations that can be found in
Section 0.5 and on page 73 of \cite{bertoin96} to see that the
Laplace transform of $\sigma$ is
  \begin{equation*}
    \begin{split}
      E\exp(-\lambda \sigma_t)&=\exp\left\{ -t\int_0^\infty(1-\e^{-\lambda x})\Pi(dx)
      \right\}\\
      &=\exp\left\{ -ct\int_0^\infty(1-\e^{-\lambda x})x^{-p/2-1}dx \right\}\\
      &=\exp(-ct\lambda^{p/2}).
    \end{split}
  \end{equation*}
Therefore $\sigma$ is stable with index $p/2$.
\end{proof}

\section{Construction of a Fleming-Viot process}
\label{FVproc}

We recall the following informal description of a Fleming-Viot-type
particle system from \cite{burdzymarch00}. Consider an open set
$D\subset\R^d$ and an integer $N\geq 2$. Let
$\X_t=(X_t^1,\dotsc,X_t^N)$ be a process with values in $D^N$ defined
as follows. Let $\X_0=(x^1,\dotsc,x^N)\in D^N$. Then the processes
$X_t^1,\dotsc,X_t^N$ evolve as independent Brownian motions until the
time $\tau_1$ when one of them, say, $X^j$ hits the boundary of $D$.
At this time one of the remaining particles is chosen uniformly, say,
$X^k$, and the process $X^j$ jumps at time $\tau_1$ to
$X^k_{\tau_1}$. The processes $X_t^1,\dotsc,X_t^N$ continue evolving
as independent Brownian motions after time $\tau_1$ until the first
time $\tau_2>\tau_1$ when one of them hits the boundary of $D$. Again
at the time $\tau_2$ the particle which approaches the boundary jumps
to the current location of a particle chosen uniformly at random from
amongst the ones strictly inside $D$. The subsequent evolution of
$\X$ proceeds in the same way. The above recipe defines the process
$\X_t$ only for $t \leq \tau_\infty = \lim_{k\to \infty} \tau_k$.
There is no natural way to define the process $\X_t$ for $t>
\tau_\infty$. Hence, it is a natural problem to determine whether
$\tau_\infty=\infty$, a.s.

\begin{theorem}\label{thm:exist}
There exists a constant $c=c(N,d)$ such that if $D\subset\R^d$
is a bounded Lipschitz domain with the Lipschitz constant
$L<c(N,d)$, then $\tau_\infty=\infty$, a.s. Moreover, $c(N,d)$
increases in $N$, decreases in $d$ and
 \begin{equation}\label{eq:limit}
   \lim_{N\to\infty}c(N,d)=c(d)=\frac{1}{\sqrt{d-1}}.
 \end{equation}
\end{theorem}

\begin{proof}
First note that $\tau_\infty$ is finite if and only if all the
processes $X_t^1,\dotsc,X_t^N$ hit $\partial D$ at the same
time, so we need to prove that this is impossible. The idea of
the proof is to construct processes $Y_t^1,\dotsc,Y_t^N$ which
are easy to analyze, with values in $\ol D$ and such that for
every $1\leq j\leq N$,
\begin{equation}\label{eq:property}
  \set{t:X_t^j\in\partial D}
  \subset\set{t:Y_t^j\in\partial D}\overset{\text{df}}{=}A_j.
\end{equation}
Then we will prove that that $A_1\cap \dotsc \cap A_N=
\emptyset$, a.s.

Recall the definition of $\theta_{p,d}$ and $K_{p,d}$. Let $p'
= 2 - 2/N$ and
\begin{equation*}
  c(N,d)=\cot\theta_{p',d}.
\end{equation*}

Define
\begin{equation*}
  D_r=\set{x\in D:\dist(x,\partial D)\geq r}.
\end{equation*}
Since $D$ is Lipschitz, there exists a small $r>0$ for which
the following is true. For every $x\in \ol{D\setminus D_r}$
there exist an orthonormal coordinate system $CS_x$,
$y_x\in\partial D$ and a Lipschitz function $F_x:\R^{d-1}\to\R$
such that $y_x$ is the origin of $CS_x$ and
\begin{equation*}
  D\cap B(y_x,r)\subset
  \set{y\;{\rm  in  }\;CS_x:y^d>F_x(\wt y)} \cap B(y_x,r).
\end{equation*}
Moreover, since $L<c(N,d)$, we can choose $y_x$ and $CS_x$ so
that we can find a cone $ K_x$ with vertex $y_x$ and axis
passing through $x$ which can be described in $CS_x$ as
$K_{p,d}$ with $p<p'$, and such that
\begin{equation}\label{eq:stozki}
  K_{p',d}\cap B(y_x,r)\subset K_{p,d}\cap B(y_x,r)\subset D\cap B(y_x,r).
\end{equation}

Next we will present a very special construction of the process
$\X_t$, based on a family of independent Brownian motions. We
need this construction to show independence of processes $Y^1,
Y^2, \dots, Y^N$, to be constructed in a subsequent step.

We fix $j\in\set{1,\dotsc,N}$ and let $\tau_k^j$ denote the
time of the $k$-th jump of $X_t^j$. We represent the evolution
of $X^j$ on the interval $[0,\tau_1^j)$ as follows. We start
with a family of independent Brownian motions $\wt W^0, \wt
W^1, \wt W^2,\dotsc$ in $\R^d$ starting from 0. Suppose that
$X^j_0= x_0\in D\setminus D_{r/2}$. The argument needs only
minor modifications if $x_0\in D_{r/2}$. Let $W^0= \wt W^0 +
x_0$ and consider the cone $ K_{x_0}$ defined as above. Let
\begin{align*}
  \sigma'_1&=\inf\set{t>0:W^0_t\in\partial K_{x_0}},\\
  \sigma''_1&=\sigma'_1\wedge\inf\set{t>0:W_t^0 \notin D_r},\\
  \sigma_1&=
  \begin{cases}
    \sigma'_1,&\text{if $\sigma''_1=\sigma'_1$,}\\
    \inf\set{t>\sigma''_1:W_t^0\not\in D_{r/2}},&\text{if $\sigma''_1<\sigma'_1$.}
  \end{cases}
\end{align*}
Inductively, for $n\geq 1$, given $x_n=W^{n-1}_{\sigma_n}$ we
define the cone $ K_{x_n}$, let $W^n = \wt W^n + x_n$, and
\begin{align*}
  \sigma'_{n+1}&=\inf\set{t>0:W^n_t\in\partial K_{x_n}},\\
  \sigma''_{n+1}&=\sigma'_{n+1}\wedge\inf\set{t>0:W_t^n\notin D_r},\\
  \sigma_{n+1}&=
  \begin{cases}
    \sigma'_{n+1},&\text{if $\sigma''_{n+1}=\sigma'_{n+1}$,}\\
    \inf\set{t>\sigma''_{n+1}:W_t^n\not\in D_r},&\text{if $\sigma''_{n+1}<\sigma'_{n+1}$.}
  \end{cases}
\end{align*}
Then we define $T_0=0$,
\begin{equation*}
  T_n=\tau_1^j\wedge\sum_{k=1}^n\sigma_k,\quad n\geq 1,
\end{equation*}
and
\begin{equation*}
  X_t^j=W_t^k,\quad\text{for $T_{k-1}<t\leq T_k$},\quad k=0,1,2,\dotsc.
\end{equation*}
This procedure represents $X^j$ on the interval $[0,
\tau_1^j)$. Strictly speaking, $\tau^j_1$ is defined in terms
of $\wt W^0, \wt W^1, \wt W^2,\dotsc$.

The complete construction of $\X_t$ on the interval $[0,
\tau_\infty)$ requires that we start with a family $\{\wt
W^{j,k,n}\}_{1\leq j \leq N, k\geq 0, n \geq 0}$ of independent
Brownian motions. For fixed $j$ and $k$, the subfamily $\{\wt
W^{j,k,n}\}_{ n \geq 0}$ is used to construct $X^j$ on the
interval $[\tau_k^j,\tau_{k+1}^j)$, according to the recipe
described above. The whole procedure is straightforward and
elementary but tedious to describe so we leave the details to
the reader.

Let $\{\wt Z^{j,k,m,n}\}_{1\leq j \leq N, k\geq 0, m \geq 0,
n\geq 0}$ be a family of independent copies of the process $Z$
defined in \eqref{eq:a31-1}, independent of $\{\wt
W^{j,k,n}\}_{1\leq j \leq N, k\geq 0, n \geq 0}$. We will
present a construction of $Y^j$ on the interval $[T_0, T_1)
\subset [0, \tau^j_1)$, based on $X^j$ and $\{\wt
Z^{j,0,0,n}\}_{n \geq 0}$. For any fixed $j,k$ and $m$, we can
construct $Y^j$ on $[T_m, T_{m+1}) \subset
[\tau_k^j,\tau_{k+1}^j)$ using $X^j$ and $\{\wt Z^{j,k,m,n}\}
_{n \geq 0}$ in an analogous way. (Strictly speaking, we should
not speak about $[T_m, T_{m+1}) \subset
[\tau_k^j,\tau_{k+1}^j)$ but about a subinterval of
$[\tau_k^j,\tau_{k+1}^j)$ constructed in a way analogous to
$[T_m, T_{m+1})$.)

Let $\T_n$ be the isometry that maps $K_{p,d}$ onto $K_{x_n}$
and let $Z^n = \T(\wt Z^{j,0,0,n})$. We introduce a moving cone
$C^n_t$ with the vertex $X^j_t$ and the axis parallel to the
axis of $ K_{x_n}$, but directed downwards, i.e., we put
$C_*^n(x)=x- K_{x_n}$ and $C^n_t=C^n_*(X_t^j)$, $T_n\leq t <
T_{n+1}$. Let
\begin{gather*}
  S_1=\sigma''_1\wedge\inf\set{t>0:\Z{1}_t\in\partial C^0_t},\\
  Y^{(1)}_t=X_t-X(S_1)+\Z{1}(S_1),\quad t\geq S_1,\\
  R_1=\sigma''_1\wedge\inf\set{t>S_1:Y^{(1)}_t\in\partial
  K_{x_0}}.
\end{gather*}
By definition, $S_1\leq T_1$. Then we define
\begin{equation*}
  Y^j_t=
  \begin{cases}
    \Z{1}_t,&\text{for $t\in[0,S_1)$,}\\
    Y^{(1)}_t,&\text{for $t\in[S_1,R_1)$}.
  \end{cases}
\end{equation*}
For $n\geq 2$, we define
\begin{gather*}
  S_n=\sigma''_1\wedge\inf\set{t>R_{n-1}:\Z{n}\left( t-R_{n-1} \right)
  \in\partial C^{n-1}_t},\\
  Y^{(n)}_t=X_t-X(S_n)+\Z{n}\left( S_n-R_{n-1} \right),\quad t\geq S_n,\\
  R_n=\sigma''_1\wedge\inf\set{t>S_n:Y^{(n)}_t\in\partial\bar
  K_{x_0}},
\end{gather*}
and
\begin{equation*}
  Y^j_t=
  \begin{cases}
    \Z{n}\left( t-R_{n-1} \right),&\text{for $t\in[R_{n-1},S_n)$,}\\
    Y^{(n)}_t,&\text{for $t\in[S_n,R_n)$}.
  \end{cases}
\end{equation*}
We continue this process until the time $\sigma''_1$ and then we put
$Y_t^j=Y^j_{\sigma''_1}$ for $t\in[\sigma''_1,\sigma_1)$.

By construction, processes $X^1,\dotsc,X^N$ and
$Y^1,\dotsc,Y^N$ satisfy \eqref{eq:property}. Moreover,
independence of all processes in the family $\{\wt
Z^{j,k,m,n}\}_{1\leq j \leq N, k\geq 0, m \geq 0, n\geq 0} \cup
\{\wt W^{j,k,n}\}_{1\leq j \leq N, k\geq 0, n \geq 0}$ implies
that processes $Y^1,\dotsc,Y^N$ are independent. It remains to
prove that, a.s.,
\begin{equation*}
  A_1\cap\dotsc\cap A_N=\emptyset.
\end{equation*}

Recall that $A_j = \set{t\geq 0: Y^j \in \prt D}$. The
construction of $Y^j$ from independent pieces of processes
analogous to $Z$ suggests that $A_j$ is the range of a stable
subordinator, because of Lemma \ref{lem:stable}. The matter is
slightly complicated by the fact that $Y_t^j=Y^j_{\sigma''_1}$
for $t\in[\sigma''_1,\sigma_1)$, and similarly for other
analogous intervals. To deal with this problem, we introduce
the following sequence of stopping times, $U_0 = 0$,
\begin{align*}
U^*_k & = \inf\{ t\geq U_k:
\max_{1\leq j \leq N} \dist(X^j_t , \prt D) \geq r\},
\qquad k \geq 0, \\
U_k & = \inf\{ t\geq U^*_{k-1}:
\max_{1\leq j \leq N} \dist(X^j_t , \prt D) \leq r/2\},
\qquad k \geq 1.
\end{align*}
On each interval $(U_k, U^*_k]$, sets $A_1, \dots, A_N$ are
independent and each one has the same distribution as the range
of a stable subordinator with index $p/2$. It will suffice to
prove that for each fixed $k$, a.s.,
\begin{equation}\label{eq:empty}
(U_k, U^*_k] \cap  A_1\cap\dotsc\cap A_N=\emptyset.
\end{equation}
We use the following result of Hawkes \cite{hawkes77}: The
ranges of two independent stable subordinators with indices
$\alpha$ and $\beta$ intersect if and only if $\alpha+\beta>1$,
in which case the intersection is stochastically equivalent to
the range of a stable subordinator of index $\alpha+\beta-1$.
Therefore, by induction, \eqref{eq:empty} holds if and only if
$\frac{Np}{2}-N+1<0$. This condition holds since $p< p' =
2-\frac{2}{N}$.
\end{proof}

\begin{remark}\label{rem:1}
Let $D\subset\R^d$ be a bounded Lipschitz domain with the
Lipschitz constant $L<\frac{1}{\sqrt{d-1}}$. Then by
\eqref{eq:limit} we see that there exists $N_0$ so large that
$L<c(N,d)$ for all $N\geq N_0$. In consequence, the
Fleming-Viot-type particle process $\X_t$ in $D$ is well
defined for all $t\geq 0$ provided it consists of $N$ particles
with $N\geq N_0$.
\end{remark}

\begin{example}\label{ex:sep2}
The proof of Theorem 1.1 in \cite{burdzymarch00} contains an error.
Formula (2.1) in \cite{burdzymarch00} does not follow ``by
induction'' from the previous statement. We will show that the error
is irreparable in the following sense. The proof of Theorem 1.1 in
\cite{burdzymarch00} is based only on two properties of Brownian
motion--- the strong Markov property and the fact the the hitting
time distribution of a compact set has no atoms (assuming that the
starting point lies outside the set). Hence, if some version of that
argument were true, it would apply to almost all non-trivial examples
of Markov processes with continuous time, and in particular to all
diffusions.  However we may find a diffusion for which the analogue
of Theorem 1.1 in \cite{burdzymarch00} is false. Let $X_t$ be the
diffusion  on $\halfint{0,\infty}$,  started at $X_0 = 1$ and
satisfying the SDE
\begin{equation*}
dX_t = dW_t - \frac 5{2X_t}dt.
\end{equation*}
We make $0$ absorbing so that it can play the role of the boundary for the domain $D= \br{0,\infty}$. Notice that although $X_t$ is not a Bessel process, as we have reversed the drift term, it scales in the same way. That is, for $\alpha>0$, $\alpha X_{t\alpha^{-2} }$ is a diffusion satisfying the same SDE, but started at $\alpha$. Let $\Y^i_t = \br{Y^{i,1}_t,Y^{i,2}_t},\, i=1\dots\infty$, be a double sequence of independent copies of $X_t$, and set
\begin{align*}
\sigma_i &= \inf\setof{t>0}{Y^{i,1}_t \wedge Y^{i,2}_t = 0}, \\
\alpha_i &= Y^{i,1}_{\sigma_i} \vee Y^{i,2}_{\sigma_i}.
\end{align*}
Now, construct a two-particle Fleming-Viot type process $\X_t =
\br{X^1_t, X^2_t}$ as follows. First let $\tau_1=\sigma_1$ and set
$\X_t = \Y^1_t$ for $t\in\halfint{0,\tau_1}$. At $\tau_1$ one of the
particles hits the boundary and jumps to $\xi_1 = \alpha_1$. To
continue the process we use the scaling property of $\Y_t$ and set
$\X_t = \xi_1\Y^2_{\br{t-\tau_1}\xi_1^{-2}}$ for
$t\in\halfint{\tau_1,\tau_2}$ where $\tau_2 = \tau_1 +
{\xi_1}^2\sigma_2$. At $\tau_2$ a second particle hits the boundary
and jumps, this time to $\xi_2 = \alpha_2\xi_1$, and we continue the
process in the same way by setting
\begin{align*}
\xi_i &=\prod_{j=1}^i \alpha_j, \qquad
\tau_i = \sum_{j=1}^{i} {\xi_{j-1}}^2 \sigma_{j}, \\
\X_t &= \xi_i \Y^i_{\br{t-\tau_i}\xi_i^{-2}} ,
\qquad \text{for } t\in\halfint{\tau_i,\tau_{i+1}}.
\end{align*}

Then $\X_t$ evolves as two independent copies of $X_t$ with
Fleming-Viot type jumps when a particle hits the boundary. The
process $\X_t$ is well defined up until $\tau_\infty$ and if the
analogue of \cite[Theorem 1.1]{burdzymarch00} were to hold for this
process we would have $\tau_\infty=\infty$ almost surely. In fact the
opposite is true. We will show now that $\E\tau_\infty<\infty$ and
hence $\tau_\infty<\infty$ almost surely. To do this it will be
sufficient to show $\E\br{{\alpha_1}^2}<1$ and $\E{\sigma_1}<\infty$.
Let $f\br{x,y} = x^4 +y^4 - x^2y^2$ and notice $f(x,x) = f(x,0) =
f(0,x) = x^4$.  We may check using Ito's formula that
$f\br{Y^{i,1}_{t\wedge\sigma_i}, Y^{i,2}_{t\wedge\sigma_i}}$ is a
positive local martingale and hence a supermartingale.  By the
optional stopping theorem
\begin{align*}
\E\br{{\alpha_1}^4} = \E{f\br{Y^{1,1}_{\sigma_1}, Y^{1,2}_{\sigma_1}}}
\leq\E{f\br{Y^{1,1}_{0}, Y^{1,2}_{0}}  }
= 1.
\end{align*}
Furthermore, $\alpha_1$ is not almost surely constant and so by
Jensen's inequality
\begin{equation*}
\E\br{{\alpha_1}^2} < \sqrt{\E\br{{\alpha_1}^4}} = 1.
\end{equation*}

We may use Ito's formula again to show that ${X_t}^2 +4t$ is a local
martingale and so by the optional stopping theorem again we have that
$\E\br{\sigma_1}\leq\frac 14$.

By independence of the $\Y^i$ processes we have that
$\E\br{{\xi_i}^2} = \E\br{{\alpha_1}^2}^i$ and so
\begin{align*}
\E{\tau_\infty} = \sum_{j=1}^{\infty} \E\br{{\xi_{j-1}}^2 \sigma_{j} }
\leq \frac 14 \sum_{j=0}^{\infty} \E\br{{\alpha_1}^2}^{j}  <\infty.
\end{align*}

\end{example}

\section{Hitting probabilities of compact sets}

This section is devoted to a technical estimate needed in the
proof of Theorem \ref{thm:A14.2}. Recall definitions of $D_r$
and $\X_t=(X_t^1,\dotsc,X_t^N)$.

\begin{lemma}\label{lem:estimate}
Fix $N\geq 2$ and let $D\subset\R^d$ be a bounded Lipschitz
domain with the Lipschitz constant $L<c(N,d)$.
  \begin{enumerate}[(i)]
    \item For any fixed $k\in\set{1,\dotsc,N}$, and for
        every $r>0$ such that $\inter D_r \ne \emptyset$,
        there exist $c>0$ and $t>0$ such that for all
        $\bx\in D^N$,
      \begin{equation*}
    \P[\bx]\left( X^k_{t}\in D_r \right)\geq c.
      \end{equation*}
    \item For every $r>0$ such that $\inter D_r \ne \emptyset$,
        there exist $c>0$  and $t>0$ such that for all
        $\mathbf{x}\in D^N$,
      \begin{equation*}
    \P[\mathbf{x}]\left( \X_{t}\in D_r^N \right)\geq c.
      \end{equation*}
  \end{enumerate}
\end{lemma}

\begin{proof} (i)
Fix $r>0$ such that $\inter D_r \ne \emptyset$. Recall that
notation such as $T_{D_r}$, $T_{\prt D}$, etc. refers to
hitting times by Brownian motion. By Theorem \ref{thm:main}
there exists $c_0=c_0(r)$ such that for all $x\in D$,
  \begin{equation}\label{eq:a31-3}
    \P[x]\left( T_{D_r}<T_{\partial D} \right)\geq c_0 \E[x]T_{\partial D}.
  \end{equation}
Fix $k$ and let $T^{X^k}_{D_r} = \inf\set{t \geq 0: X^k_t \in
D_r}$, and
\begin{equation*}
    Y_t=X^k(t\wedge T^{X^k}_{D_r}).
  \end{equation*}
Define $T_0=0$ and
  \begin{equation*}
    T_{n+1}=\inf \left\{ t>T_n:\lim_{s\to t^-}Y_s\in\partial D \right\}
    \land T^{X^k}_{D_r}.
  \end{equation*}
Let $M_0=0$ and
  \begin{equation*}
    M_n=\frac{1}{c_0}\I_{\set{Y(T_n)\in D_r}}-T_n,\quad n\geq 1,
  \end{equation*}
  and
  \begin{equation*}
    \F_n=\sigma(\X_t, t\leq T_n).
  \end{equation*}
It is easy to see that $E T_n < \infty$ so $E|M_n| < \infty$.
For $\bx=(x_1, x_2, \dots, x_N)\in D^N$ with $x_k\not\in D_r$,
  \begin{align*}
      &\E[\x]\left( M_{n+1}-M_n\mid\F_n \right)\\
&= \frac{1}{c_0} \E[\x]\left(\I_{\set{Y(T_{n+1})\in D_r}}
(\I_{\set{Y(T_n)\notin D_r}} + \I_{\set{Y(T_n)\in D_r}})
- \I_{\set{Y(T_n)\in D_r}} \mid \F_n \right)
-\E[\x]\left( T_{n+1}-T_n\mid \F_n \right)\\
&= \frac{1}{c_0} \E[\x]\left(\I_{\set{Y(T_{n+1})\in D_r}}
\I_{\set{Y(T_n)\notin D_r}} + \I_{\set{Y(T_n)\in D_r}}
- \I_{\set{Y(T_n)\in D_r}} \mid \F_n \right)
-\E[\x]\left( T_{n+1}-T_n\mid \F_n \right)\\
      &=\frac{1}{c_0}\I_{\set{Y(T_n)\not\in D_r}}\P[\x]\left(
      Y(T_{n+1})\in D_r\mid \F_n\right)
      -\E[\x]\left( T_{n+1}-T_n\mid \F_n \right).
  \end{align*}
We have on the event $\{Y(T_n)\not\in D_r\}$,
  \begin{equation*}
    \E[\x]\left( M_{n+1}-M_n\mid\F_n
    \right)
    \geq \frac{1}{c_0}\P[X^k(T_n)]\left(T(D_r)<T_{\partial D}\right)
    -\E[X^k(T_n)]T_{\partial D}\geq 0,
  \end{equation*}
by \eqref{eq:a31-3}. On the event $\{Y(T_n)\in D_r\}$, we have
$T_{n+1}=T_n$, $Y_{T_{n+1}}\in D_r$, and so
  \begin{equation*}
    \E[\x](M_{n+1}-M_n\mid\F_n)=0.
  \end{equation*}
Combining the last two formulas, we conclude that $\set{M_n}$
is a submartingale with respect to $\set{\F_n}$.

Define
  \begin{equation*}
    S=\inf\set{j:T_j\geq 1} \land \inf\set{j: Y_{T_j} \in D_r}.
  \end{equation*}
Fix an $\bx\in D^N$ and consider two cases. First, we may have
\begin{align*}
    \P[\x](S = \inf\set{j: Y_{T_j} \in D_r}) \geq 1/2.
\end{align*}
In this case,
\begin{align}\label{eq:a31-4}
    \P[\x](T^{X^k}_{D_r} \leq 1) \geq 1/2.
\end{align}

The second case is when
\begin{align*}
    \P[\x](S = \inf\set{j: Y_{T_j} \in D_r}) < 1/2.
\end{align*}
In this case, $\P[\x]( S\geq 1) \geq 1/2$, so $\E[\x] T_S \geq
1/2$. The submartingale $M_n$ is bounded above by $1/c_0$ so we
can apply the optional stopping theorem to obtain
  \begin{equation*}
    \E[\x]M_S\geq \E[\x]M_0=0.
  \end{equation*}
Hence
  \begin{equation}\label{eq:s2-2}
    \P[\x]\left( Y_{T_S}\in D_r\right)\geq c_0\E[\x]T_S\geq c_0/2.
  \end{equation}
We will show that for some $t_0$,
  \begin{equation}\label{eq:aux2}
    \P[\x]\left(T^{X^k}_{D_r}\leq t_0\right)\geq c_0/4.
  \end{equation}
If $T_S > s_0$ for some $s_0 >1$ then $X^k_t$ must not hit $D_r
\cup \prt D$ for $t\in(1, s_0)$. The probability of this event
is bounded above by the probability of the event that Brownian
motion starting from $X^k_1$ will not leave the ball $B(X^k_1,
2 \diam(D))$ for $s_0-1$ units of time. The last probability is
$c_1<1$, depending on $s_0 >1$, but not depending on $X^k_1$.
By the Markov property,
\begin{align*}
    \sup_{\bx\in D^N}\P[\x]\left( T_S>s_0 \right) \leq c_1 <  1.
\end{align*}
Applying the Markov property repeatedly at times $s_0, 2s_0, \dots$,
we obtain for any $\bx\in D^N$,
  \begin{equation*}
    \P[\x](T_S>ns_0)\leq c_1^n.
  \end{equation*}
We choose $n$ so large that $c_1^n \leq c_0/4$ and let
$t_0=ns_0$. Then for $\bx\in D^N$,
  \begin{equation}\label{eq:s2-1}
    \P[\x](T_S>t_0)\leq c_0/4.
  \end{equation}
We use \eqref{eq:s2-2} and \eqref{eq:s2-1} to see that
  \begin{align*}
      c_0/2 & \leq\P[\x]\left( Y_{T_S}\in D_r\right)\\
      &=\P[\x]\left( Y_{T_S}\in D_r,T_S>
      t_0\right) +\P[\x]\left( Y_{T_S}\in D_r,T_S\leq t_0\right)\\
      &\leq\P[\x](T_S>t_0) +\P[\x]\left(T^{X^k}_{D_r} \leq t_0\right)\\
      &\leq c_0 /4 +\P[\x]\left(T^{X^k}_{D_r} \leq t_0\right).
  \end{align*}
This implies \eqref{eq:aux2}. We combine the two cases, that
is, \eqref{eq:a31-4} and \eqref{eq:aux2}, to see that for some
$t_1 < \infty$ and $c_2$, for all $\bx\in D^N$,
  \begin{equation}\label{eq:s2-3}
    \P[\x]\left(T^{X^k}_{D_r}\leq t_1\right)\geq c_2.
  \end{equation}

Let $r_1$ be such that $0< r< r_1$ and $\inter D_{r_1} \ne
\emptyset$. Let $t_2$ and $c_3$ be such that \eqref{eq:s2-3} holds
with $r_1, t_2$ and $c_3$ in place of $r,t_1$ and $c_2$, i.e.,
  \begin{equation}\label{eq:s2-4}
    \P[\x]\left(T^{X^k}_{D_{r_1}}\leq t_2\right)\geq c_3.
  \end{equation}
Let $r_2=(r_1-r)/2$ and $p_1=\P[0]\left( T_{\prt B(0,r_2)}\geq
t_2 \right)>0$. By translation invariance of Brownian motion,
$p_1=\P[y]\left( T_{\prt B(y,r_2)}\geq t_2 \right)$ for every
$y$. If the process $X^k$ hits $D_{r_1}$ before time $t_2$ and
then stays in the ball $B(X^k(T^{X^k}_{D_{r_1}}), r_2)$ for at
least $t_2$ units of time then $X^k$ will be inside $D_r$ at
time $t_2$. By the strong Markov property applied at the
stopping time $T^{X^k}_{D_{r_1}}$, we obtain, using
\eqref{eq:s2-4}, for all $\bx\in D^N$,
\begin{align}\label{eq:s2-5}
\P[\x](X^k_{t_2}\in
D_{r}) \geq p_1 \P[\x]\left(T^{X^k}_{D_{r_1}}\leq t_2\right)\geq
p_1 c_3 >0.
\end{align}
This completes the proof of part (i) of the lemma.

(ii) Recall that $r>0$ is fixed and such that $\inter D_r \ne
\emptyset$. Let $r_3$ and $r_4$ be such that $0< r < r_3 < r_4$
and $\inter D_{r_4} \ne \emptyset$. Let $r_5 = \min(r_3 - r,
r_4-r_3)/2$. We choose $t_3$ and $c_4$ so that \eqref{eq:s2-5}
can be applied with $r_4$ in place of $r$,
\begin{align*}
\P[\x](X^k_{t_3}\in D_{r_4}) \geq c_4 >0.
\end{align*}
Let $p_2 = \inf_{y\in D} \P[y](T_{\prt D} \leq t_3) $ and note
that $p_2 > 0$. Let $p_3=\P[y]\left( T_{\prt B(y,r_5)}\geq 2t_3
\right)>0$ and note that $p_3$ does not depend on $y$.

Let $A$ be the intersection of the following events.

(a) The process $X^1$ is in $D_{r_4}$ at time $t_3$, and it
stays in $B(X^1_{t_3}, r_5)$ for all $t\in[t_3,3 t_3]$.

(b) For every $j=2, \dots, N$, the process $X^j$ jumps at a
time $s_j\in [t_3, 2t_3]$ to $X^1_{s_j}$, and then stays in the
ball $B(X^j_{s_j}, r_5) = B(X^1_{s_j}, r_5)$ for all
$t\in[s_j,s_j+2 t_3]$.

By the strong Markov property and the definition of the process
$\X$, the probability of $A$ is bounded below by $c_5 = c_4 p_3
(p_2 (1/(N-1)) p_3)^{N-1}$. If $A$ occurs then $\X_{3t_3}\in
D_{r}^N$. Hence, for every $\x\in D^N$,
  \begin{equation*}
    \P[\x]\left( \X_{3t_3}\in D_{r}^N \right)\geq c_5>0.
  \end{equation*}
  This proves part (ii) of the lemma.
\end{proof}

\section{Stationary distribution for the particle system}
\label{sec:station}

The two theorems proved in this section generalize the
analogous results in \cite{burdzymarch00}, where the proofs
were given only for domains satisfying the internal ball
condition.

\begin{theorem}\label{thm:A14.2}
Suppose that $D\subset\R^d$ is a bounded Lipschitz domain with
the Lipschitz constant $L<c(N,d)$, where $c(N,d)$ is as in
Theorem \ref{thm:exist}. Then there exists a unique stationary
probability distribution $\M^N$ for $\X_t$. The process $\X_t$
converges to its stationary distribution exponentially fast,
i.e., there exists $\lambda>0$ such that for every $A\subset
D^N$,
  \begin{equation}\label{eq:s2-7}
    \lim_{t\to\infty}\e^{\lambda t}
    \sup_{\x\in D^N}\left|\P[\mathbf x]\left( \X_t\in A
    \right)-\M^N(A)\right|=0.
  \end{equation}
\end{theorem}

\begin{proof}
We have shown in Lemma \ref{lem:estimate} (ii) that for any
$r>0$, with probability higher than $p_0=p_0(r)>0$, the process
$\X_t$ can reach the compact set $D^N_{r}$ within $t_0>0$ units
of time. This and the strong Markov property applied at times
$2t_0,4t_0,6t_0,\dotsc$ show that the hitting time of $D^N_r$
is stochastically bounded by an exponential random variable
with the expectation independent of the starting point of
$\X_t$.  Since the transition densities $p^\X_t({\bf x}, {\bf
y})$ for $\X_t$ are bounded below by the densities for the
Brownian motion killed at the exit time from $D^N$, we see that
$p^\X_t({\bf x}, {\bf y})> c_1 >0$ for ${\bf x}, {\bf y} \in
D^N_{r}$.  Fix arbitrarily small $s>0$ and consider the
``skeleton'' $\{\X_{ns}\}_{n\geq 0}$. The properties listed in
this paragraph imply that the skeleton has a stationary
probability distribution and that it converges to that
distribution exponentially fast, i.e., \eqref{eq:s2-7} holds
for the skeleton, by Theorem 2.1 in \cite{downmeyn95} or
Theorem 16.0.2 (ii) and (vi) of \cite{meyntweedie}. See the
proof of Proposition 1.2 in \cite{traps} for an argument
showing how to pass from the the statement of uniform
ergodicity for the skeleton to the analogous statement for the
continuous process $t\to \X_t$. We sketch this argument here.
Take any $\eps>0$ and find $t_1= n_1 s$ such that
  \begin{equation}\label{eq:s2-8}
    \e^{\lambda t}
    \sup_{\x\in D^N}\left|\P[\mathbf x]\left( \X_t\in A
    \right)-\M^N(A)\right|\leq \eps
  \end{equation}
holds for $t \geq t_1$ of the form $t = n s$. Consider an
arbitrary $t_2 > t_1$, not necessarily of the form $ns$. Let
$m$ be the integer part of $t_2/s$ and let $u = t_2 - ms$. Note
that $m\geq n_1$. Since \eqref{eq:s2-8} holds for $t=ms$, the
semigroup property applied at time $u$ shows that
\eqref{eq:s2-8} holds also at time $t_2$.
\end{proof}

\begin{theorem}
Suppose that $D$ is a bounded Lipschitz domain with the
Lipschitz constant $L<\frac{1}{\sqrt{d-1}}$. For $N\geq N_0$
(see Remark \ref{rem:1}) let $\XX_\M^N$ be the stationary
empirical measure. Let $\varphi$ be the first eigenfunction for
Laplacian in $D$ with the Dirichlet boundary conditions,
normalized so that $\int_D\varphi=1$. Then the sequence of
random measures $\XX_\M^N$, $N\geq N_0$, converges as
$N\to\infty$ to the (non-random) measure with the density
$\varphi$, in the sense of weak convergence of random measures.
\end{theorem}

\begin{proof}
Recall processes $Y^j$ defined in the proof of Theorem
\ref{thm:exist}. By construction, we have $\dist(Y^j_t, \prt D)
\leq \dist(X^j_t, \prt D)$, for all $j$ and $t$.

It is elementary to see that the process $Z$ constructed in
Section \ref{sec:Zproc} has the property that
\begin{align*}
    \lim_{r\downarrow 0} \limsup_{t\to\infty} \frac1t
    \int_0^t \I_{\{\dist(Z_s, \prt D) \leq r\}} ds =0, \ \text{a.s.}
\end{align*}
In view of the construction of $Y^j$ from independent copies of
$Z$, we also have, for every $j$,
\begin{align*}
    \lim_{r\downarrow 0} \limsup_{t\to\infty} \frac1t
    \int_0^t \I_{\{\dist(Y^j_s, \prt D) \leq r\}} ds =0, \ \text{a.s.}
\end{align*}
Hence, for every $j$,
\begin{align*}
    \lim_{r\downarrow 0} \limsup_{t\to\infty} \frac1t
    \int_0^t \I_{\{\dist(X^j_s, \prt D) \leq r\}} ds =0, \ \text{a.s.}
\end{align*}
This implies that for every $p_1>0$, one can find $r>0$ so
small that if $\X$ has the stationary measure $\M^N$ then for
every $t$, $\P(X^j_t \notin D_r)\leq p_1$. It follows that for
any $N$, the mean measure $E \XX^N_\M$ of the compact set
$D_{r}$ is not less than $1-p_1$.  Hence, the mean measures $E
\XX^N_\M$ are tight in $D$.  Lemma 3.2.7, p.~32, of
\cite{dawson92} implies that the sequence of random measures
$\XX^N_\M$ is tight and so it contains a convergent
subsequence.

One can complete the proof of the claim that the random
measures $\XX_\M^N$ converge as $N\to\infty$ to the measure
with the density $\varphi$ exactly as in the proof of Theorem
1.4 in \cite{burdzymarch00}, starting on line 9 of page 699.
\end{proof}

\section{Polyhedral domains}

In this section we show that the Lipschitz constant $c\br{N,d}$ in
Theorem \ref{thm:exist} is not sharp, that is, $\tau_\infty =
\infty$, a.s., in some Lipschtz domains with arbitrarily large
Lipschitz constant. Specifically, we will demonstrate the existence
of the two particle process  for all times in arbitrary polyhedral
domains. Unfortunately, our method cannot be easily adapted to the
multiparticle case, so we leave this generalization as an open
problem.

\begin{definition}
We say an open set $D\subset\R^d$ is a \emph{polyhedral domain} if
there exist simplicial complexes $\calig K \supset\partial \calig
K$ such that $\overline D = \bm{\calig K}$ and $\partial D =
\bm{\partial\calig K}$.
\end{definition}
For the remainder of this section we will assume that $D =
\inter\bm{\calig K}$ is a polyhedral domain. Let $\X_t =
\br{X^1_t,X^2_t}$ be a Fleming-Viot process in $D$ and define jump
times $\tau_i$ as before. We will show:
\begin{theorem}\label{sam:thm:main}
If $D$ is a polyhedral domain and $\X_t=\br{X^1_t,X^2_t}$ is a
Fleming-Viot process with jump times $\tau_i$ then
$\tau_i\to\infty$ as $i\to\infty$ almost surely.
\end{theorem}

As $\X_t$ is a \cadlag process we have $ X^1_{\tau_i} = X^2_{\tau_i}$
for each $i\in\N$, so we may define a sequence of \emph{jump points}
$\xi_i = X^1_{\tau_i} = X^2_{\tau_i}$. Since $\overline D$ is
compact, $\xi_i$ has at least one limit point in $\overline D$. To
prove Theorem \ref{sam:thm:main} we will examine the  behavior of
$\X_t$ when both particles are close to a limit point of $\xi_i$ and,
assuming that $\tau_\infty<\infty$, arrive at a contradiction.

First we will  show that if $t\in \halfint{\tau_i,\tau_{i+1}}$
then $\X_t$ cannot stray too far from $\br{\xi_i, \xi_i}$.

\begin{lemma}\label{sam:besselness}
Set $V^1_t = \bmm{X^1_t - \xi_i},V^2_t = \bmm{X^2_t - \xi_i}$ for
$t\in\halfint{\tau_i,\tau_{i+1}}$. If $\tau_\infty<\infty$ then
$V^1_t\to 0$ and $V^2_t\to 0$ as $t\to\tau_\infty$.
\end{lemma}

\begin{proof}
It suffices to consider only $V^1_t$. Notice that $V^1_t$ is a
$d$-dimensional Bessel process ($\bes\br d$, for short), reset to $0$
at each $\tau_i$.  So setting $\Delta V^1_i =V^1_{\tau^-_i}$ we may
extract a Brownian motion
\begin{equation*}
    W_t = V^1_t + \; \sum_{\setof{i\in\N}{\tau_i\leq t}} \Delta V^1_i \;  - \int_{0}^t \frac{d-1}{2V^1_t} dt.
\end{equation*}

Consider  $\varepsilon>0$. We will count the number of upcrossings of
the interval $\bs{\frac \varepsilon 2 , \varepsilon}$ within a short
time interval $\bs{t,t+\delta}$, where $\delta =
\varepsilon^2/(4\br{d-1})$. Consider times $t<s'<s<t+\delta$ where
$V^1_s\geq\varepsilon$ and $s' = \sup\setof{\tilde s<s}{V^1_{\tilde
s} = \frac\varepsilon 2}$. Notice as $V^1$ only jumps downwards there
is no $i\in \N$ such that $s'<\tau_i\leq s$. We have
\begin{equation*}
\begin{split}
W_{s} - W_{s'}  &=   V^1_{s} - V^1_{s'}  -  \int_{s'}^{s}  \frac{d-1}{2V^1_t} dt \\
& \geq \frac\varepsilon 2 - \br{s - s'}  \frac{d-1}{\varepsilon} \\
&\geq \frac\varepsilon 2 -
\frac{\varepsilon^2}{4\br{d-1}} \frac{d-1}{\varepsilon}
=\frac\varepsilon 4.
\end{split}
\end{equation*}
So on a short time interval, each upcrossing of $\bs{\frac
\varepsilon 2,\varepsilon}$ by $V^1$ corresponds to an
oscillation of $\frac\varepsilon 4$ by $W$.  As $W$ is a
Brownian motion, with probability $1$, $V^1$ makes only
finitely many upcrossings of $\bs{\frac \varepsilon
2,\varepsilon}$ in a given time interval $\bs{t,t+\delta}$.  If
$\tau_\infty<\infty$,  we may find $n\in\N$ with
$\tau_n\geq\tau_\infty-\delta$. So if $V^1_t>\varepsilon$ for
some $\tau_n<\tau_i<t<\tau_{i+1}$ then  as $V^1$ is reset to
$0$ at $\tau_i$ there must be an upcrossing of
$\bs{\frac\varepsilon 2,\varepsilon}$ in the interval
$\halfint{\tau_i,\tau_{i+1}}\subset\bs{\tau_n,\tau_n+\delta}$.
So $V^1_t>\varepsilon$ in only finitely many intervals
$\halfint{\tau_i,\tau_{i+1}}$ and, as $\varepsilon$ is
arbitrary, $V^1_t\to 0$ as $t\to\tau_\infty.$
\end{proof}

\begin{corollary}\label{sam:corrlobus}
If $\tau_\infty<\infty$ then the sequence $\xi_i$ has no limit
point $\xi_\infty\in D$.
\end{corollary}

\begin{proof}
Fix $x\in D$. As $D$ is open, there exists some $\varepsilon$ with
$B\br{x,2\varepsilon}\subset D$. If $\xi_i\in B\br{x,\varepsilon}$
then, as both particles follow continuous paths until one exits $D$,
we must have $V^1_t\vee V^2_t>\varepsilon$ for some
$t\in[\tau_i,\tau_{i+1})$. So if $V^1_t,V^2_t\to 0$ as $t\to
\tau_\infty$ then $\xi_i\in B\br{x,\varepsilon}$ for only finitely
many $i$. As $x$ is arbitrary we see that so long as $V^1_t,V^2_t\to
0$ as $t\to \tau_\infty$, $\xi_i$ can have no limit point in~$D$.
\end{proof}

Corollary \ref{sam:corrlobus} is similar to a result in
\cite{loebus} (Step 1 of Theorem 7). In that paper, a system
consisting of an arbitrary number of particles is considered,
but the boundary $\prt D$ is assumed to be smooth.

\newcommand{\Span}[1]{\calig S_{#1}}
It is convenient at this point to introduce some notation that will
allow us to consider the behavior of $\X_t$ when it is close to the
boundary of a simplicial complex. Let $\sigma$ be a $k$-simplex with
vertices $\bc{v_0,\dots,v_k}$, that is
\begin{equation*}
    \sigma = \setof{\sum_{i=0}^k \lambda_iv_i}
    {\lambda _0,\dotsc,\lambda _k\geq 0, \,\sum_{i=0}^k\lambda_i = 1}.
\end{equation*}
Then define the \emph{interior} of $\sigma$
\begin{equation*}
    \interior\sigma = \setof{\sum_{i=0}^k \lambda_i v_i}
    {\lambda _0,\dotsc,\lambda _k> 0,\, \sum_{i=0}^k \lambda_i = 1}
\end{equation*}
and the \emph{span} of $\sigma$ to be the subspace
\begin{equation*}
    \Span\sigma = \setof{\sum_{i=0}^k \lambda_iv_i}{\sum_{i=0}^k \lambda_i = 0}.
\end{equation*}

For two simplices $\sigma_1, \sigma_2\in\calig K$ we write
$\sigma_1\leq\sigma_2$ if $\sigma_1$ is a face of $\sigma_2$ and
$\sigma_1<\sigma_2$ if $\sigma_1$ is a proper face of $\sigma_2$.
We name the \emph{star} of a simplex $\sigma$ to be the set
\newcommand{\Star}[1]{St\br{#1}}
\[\Star\sigma = \setof{\sigma_1\in \calig K}{\sigma_1\geq\sigma}\]
and define the \emph{neighborhood} of $\sigma$ as
\newcommand{\hood}[1]{\calig N\br{#1}}
\begin{equation*}
   \hood\sigma = \setof{x\in {\overline D}}
   {x\in\,\interior{\sigma_1} \text{ for some } \sigma_1 \geq \sigma}.
\end{equation*}

Given simplices  $\sigma\leq\sigma_1$ name the vertices of
$\sigma$ and $\sigma_1$,  $\bc{v_0,\dots,v_k}$ and $\bc{v_0,\dots,v_n}$
respectively. Define the \emph{wedges}
\newcommand{\Wedge}{\calig W}
\begin{align*}
\Wedge\br{\sigma,\sigma_1}
&= \setof{\sum_{i=0}^n\lambda_i v_i}
{\lambda_{k+1},\dotsc,\lambda_n>0,\; \sum_{i=0}^n \lambda_i = 1}, \\
\Wedge\br\sigma &=
\bigcup_{\sigma_1\in\Star\sigma}\Wedge\br{\sigma,\sigma_1}.
\end{align*}
Notice that $\hood\sigma\subset \overline{\Wedge\br\sigma}$ and that $\hood\sigma$ is open with respect to the subspace topology of~$\overline D$.  Notice also that $\Wedge\br\sigma$ is a product space
\newcommand{\Cone}[1]{\calig C\br{#1}}
\[\Wedge\br\sigma =  \Cone\sigma \times \Span\sigma,\]
where the cone $\Cone\sigma$ is the projection  of $\Wedge\br\sigma$
onto $\Span\sigma^\bot$.

Now, consider $\sigma\in\partial \calig K$ and  suppose there exists
a subsequence $\xi_{i_n}\to\xi_\infty\in\interior\sigma$. As
$\xi_\infty\in\interior\sigma\subset\hood\sigma$ and $\hood\sigma$ is
open in $\overline D$ we may assume without loss of generality that
$\xi_{i_n}\in\hood\sigma$ for each $n$. So consider $\X_t$ started at
$\br{\xi_{i_n}, \xi_{i_n}}$ at time $\tau_{i_n}$ and stopped at the
first time $T>\tau_{i_n}$ where one of $X^1_t, X^2_t$ exits
$\hood\sigma$. Of course, as $\hood\sigma \subset
\Wedge\br\sigma\cap\overline D$, this has the same distribution as a
Fleming-Viot process in $\Wedge\br\sigma$ started and stopped in the
same way.

So, let $\P[x]_\sigma$ and $\E[x]_\sigma$ be the probability measure
and expectation operator associated with a Fleming-Viot process in
$\Wedge\br\sigma$ started at $\X_0 = \br{x,x}$. The $\Span\sigma$
and $\Span\sigma^\bot$ components are not quite independent as
they have the same jumps, but $\P_\sigma$ allows a partial
factorization as follows.

\renewcommand{\Z}{\mathbf Z}
\begin{lemma} \label{sam:wedgelemma}
If $\X_t$ is a Fleming-Viot process in $\Wedge\br\sigma$ then there
is a well defined decomposition $\X_t = \Y_t + \Z_t$ with $\Y_t =
\br{Y^{1}_t,Y^{2}_t}\in\Cone\sigma^2,\Z_t=
\br{Z^{1}_t,Z^{2}_t}\in\Span\sigma^2$ with the following properties
\begin{itemize}
\item $\Y_t$ is a Fleming-Viot process in $\Cone\sigma$;
\item there exists a  Brownian motion $\tilde Z_t$ in $\Span\sigma$ (not
adapted to the filtration of $\X_t$), independent of $\Y_t$, such
that for each $i\in\N$ we have $\tilde Z_{\tau_i} = \zeta_i$ with
$\zeta_i = Z^1_{\tau_i} = Z^2_{\tau_i}$.
\end{itemize}
\end{lemma}

\begin{proof}
Obviously,  as $\Cone\sigma\subset \Span\sigma^\bot$, the
factorization $\X_t = \Y_t + \Z_t$ is unique. Further,  on each
interval $\halfint{\tau_i, \tau_{i+1}}$, the processes $Y^1_t$,
$Y^2_t$, $Z^1_t$ and $Z^2_t$ evolve as independent Brownian motions
on  $\Span\sigma^\bot$ and $\Span\sigma$ respectively. So as
$\Span\sigma$ is a subspace and has no boundary, $X^j_t$ jumps when
and only when $Y^j_t$ hits $\partial\Cone\sigma$, and so $\Y_t$ is
indeed a Fleming-Viot process on $\Cone\sigma$.

Now for each $i\in\N$ only one of $X^1_t$, $X^2_t$ has a
discontinuity at $\tau_{i+1}$, so there is a well defined sequence
of random variables $J_i\in \bc{1,2}$ such that $X^{J_i}_t$ is
continuous on the closed interval $\bs{\tau_i, \tau_{i+1}}$ and we
may define a continuous process
\begin{equation*}
    \tilde Z_t = Z^{J_i}_t, \quad t\in\bs{\tau_i, \tau_{i+1}}.
\end{equation*}
Then  $\tilde Z_{\tau_i} = \zeta_i$ for every $i$ and it remains to
show that $\tilde Z_t$ is a Brownian motion independent of $\Y_t$. Of
course $\tilde Z_t$ is only defined up to $\tau_\infty$.  But we may
continue $\tilde Z_t$ after  $\tau_\infty$ with an independent
Brownian motion if necessary.

Now as $\tilde Z_t$ follows either $Z^1_t$ or $Z^2_t$ then the
quadratic variation $\bigl\langle\tilde Z\bigr\rangle_t = t\mathrm I$
and, by L\'evy's characterization, we need only check that $\tilde
Z_t$ is a martingale with respect to its own natural filtration and
is independent of $\Y_t$.  Furthermore, although $\tilde Z_t$ is not
adapted to $\X_t$, for each $\tau_i$, the path $\tilde
Z\bigm|_{\bs{0,\tau_i}}$ is measurable with respect to
$\restrict\X{\bs{0,\tau_i}}$. Therefore, by the strong Markov
property, it is sufficient to consider only intervals
$\halfint{\tau_i, \tau_{i+1}}$.

In fact it suffices to consider only the first time interval
$\halfint{0,\tau_1}$. Let $\X_t$  be a Fleming-Viot process
started at  $\xi_0\in\Wedge\br\sigma$ and stopped at $\tau_1$.
Then the left limit process is a pair of independent Brownian
motions stopped at $\tau = \tau_1^-$. Set  $J= J_0$ and we have
$\xi_1 = X^J_\tau \in\Wedge\br\sigma$ and
$X^{3-J}_\tau\in\partial\Wedge\br\sigma$.

So set $\X_t = \Y_t + \Z_t$ as in the statement of the lemma and let
$\calig F^{\Y}_t$, $\calig F^{\Z}_t$ and $\calig F^{\tilde Z}_t$ be
the natural filtrations of $\Y$, $\Z$ and  $\tilde Z$ respectively.
Set $\zeta_0 = Z^1_0$, $\zeta_1 = Z^J_\tau$ to be the $\calig
F^\X_\tau$-measurable $\Z$-components of $\xi_0$ and $\xi_1$,
respectively. Thus, $\tau$ is a stopping time of $\calig F^{\Y}_t$
and $J$ is measurable with respect to $\calig F^{\Y}_\tau$. Now
crucially $\Y$ and $\Z$ are independent processes so for $t<\tau$ we
have
\begin{equation*}
    \E[\xi_0]_\sigma\br{\barrier{\zeta_1}{\calig F^\Y_\tau\vee\calig F^\Z_{t}}}
    = \E[\xi_0]_\sigma\br{\barrier{Z^J_\tau}
    {\calig F^\Y_\tau\vee\calig F^\Z_{t}}}
    = \tilde Z_t.
\end{equation*}

Thus $\tilde Z$ is a martingale,  and hence a Brownian motion, with
respect to the filtration $\calig G_t = \calig F^\Y_\tau\vee \calig
F^\Z_t$. Therefore $\tilde Z$ is independent of $\calig
F^\Y_\tau\subset\calig G_0$ and is a Brownian motion with respect to
its own natural filtration $\calig F^{\tilde Z}_t\subset\calig G_t$.

\end{proof}

Now $\Y_t$ is a process in a cone and if $\xi_i$ converges to
some point in $\interior\sigma$ then $\Y_t$ must converge to the
apex of $\Cone\sigma$.  Our next step is to show that this cannot
be the case.

\begin{lemma} \label{sam:keylemma}
If $\Y_t$ is a Fleming-Viot process in a cone $ C\subset\R^d$ then,
with probability one, $\Y_t$ does not converge to $\br{\underline 0,
\underline 0}$.
\end{lemma}

To prove this we will need to consider the angular components,
$\Phi^j_t = \frac {Y^j_t}{\bmm{Y^j_t}}$, of $\Y$. We will recall
briefly some facts about spherical Brownian motion.  We will omit
details, which can be found in \cite[Chapter 8]{oksendal},
particularly Example 8.5.8.

\newcommand{\Sp}[1][d-1]{\mathbb S^{#1}}
Let $B_t$ be a Brownian motion on $\R^d$, let the unit sphere be
denoted
\[\Sp= \setof{x\in\R^d}{\bmm x = 1},\]
and define the map $\phi:\R^d\backslash \bc{\underline 0}\to\Sp$ by
$\phi\br x = \frac x{\bmm x}$.

Now let $\Phi_t = \phi\br{B_t}$. Applying Ito's formula,
\begin{equation*}
    d\Phi_t
    = \frac 1{\bmm{B_t}} \br{I - \Phi_t\Phi^\top_t}dB_t
    - \frac{d-1}{2\bmm{B_t}^2}\Phi_t\,dt.
\end{equation*}
Note we are interpreting $\Phi_t$ as a column vector so
$\Phi_t\Phi^\top_t$ is a square matrix. Now define a differential
operator $A:\calig C^2\br{\Sp,\R}\to \calig C^0\br{\Sp,\R}$ by
\begin{equation*}
    Af\br x = \frac 12\br{\Delta f\br x - \sum_{i,j}x_i x_j\,\frac{\partial^2f}{\partial x_i\partial x_j}} - \frac{d-1}2 \sum_ix_i\frac{\partial f}{\partial x_i}.
\end{equation*}
Applying Ito's formula again, we see that $f\br{\Phi_t} - \int_0^t
\frac{Af\br{\Phi_t}}{\bmm{B_t}^2} dt$ is a local martingale for each
$f\in \calig C^2\br{\Sp,\R}$.

We may extend  this to functions of two Brownian motions by defining
$\calig A^1,\calig A^2$ by
\begin{align*}
\calig A^1 f\br{x,y} &= A\br{f\br{\cdot,y}}\br x, \\
\calig A^2 f\br{x,y} &= A\br{f\br{x,\cdot}}\br y.
\end{align*}
Then by a similar application of Ito's formula, if $B^1_t$ and
$B^2_t$ are independent Brownian motions and $\Phi^1_t =
\phi\br{B^1_t}$, $\Phi^2_t = \phi\br{B^2_t}$, $\Phi_t = (\Phi^1_t,
\Phi^2_t)$, then
\begin{equation}\label{sam:mgale}
N^f_t = f\br{\Phi^1_t,\Phi^2_t} - \int_0^t \left(\frac{\calig
A^1f\br{\Phi_t}}{\bmm{B^1_t}^2} + \frac{\calig
A^2f\br{\Phi_t}}{\bmm{B^2_t}^2} \right)dt\end{equation}
is a local martingale.

Now apply a time change to $\Phi_t$ as follows. If $\alpha\br t =
\inf\setof{s\in\R^+}{\int_0^s \bmm{B_{\tilde s}}^{-2} d\tilde s \geq
t}$, then $\Theta_t = \Phi_{\alpha\br t}$ is a Markov diffusion on
$\Sp$ with generator $A$. Let $\P[\theta_1, \theta_2]_{\Sp[]}$ and
$\E[\theta_1, \theta_2]_{\Sp[]}$ be the probability measure and
expectation operator associated with two independent copies of
$\Theta_t$ started at $\theta_1$ and $\theta_2\in \Sp$ respectively.

\begin{lemma}\label{sam:harmonic}
Let $U$ be an open subset of $\Sp$ and set
\begin{align*}
T^U_1 &= \inf\setof{t\in\R}{\Theta^1_t\in\partial U}, \\
T^U_2 &= \inf\setof{t\in\R}{\Theta^2_t\in\partial U}, \\
h^U\br{\theta_1,\theta_2} &=
\P[\theta_1,\theta_2]_{\Sp[]}\bs{T^U_1<T^U_2}.
\end{align*}
Then  $h^U\in\calig C^2\br{U^2,\R}$ and $\calig A^1h^U =
-\calig A^2h^U\geq 0$.
\end{lemma}

\begin{proof}
The process $\br{\Theta^1_t, \Theta^2_t}$ is a Markov diffusion with
generator $\calig A^1 + \calig A^2$, so by Dynkin's formula $\calig
A^1h^U + \calig A^2h^U = 0$ and it remains to show that $\calig
A^1h^U \geq 0$.

By definition of the Markov generator
\begin{align*}
\calig A^1h^U\br{\theta_1,\theta_2} &= \lim_{t\to 0}\frac{1}{t}\left(\E[\theta_1,\theta_2]_{\Sp[]} \br{h^U\br{\Theta^1_t,\theta_2}} - h^U\br{\theta_1,\theta_2}\right) \\
 &= \lim_{t\to 0}\frac{1}{t}\left(\E[\theta_1,\theta_2]_{\Sp[]}\br{\P[\Theta^1_t,\theta_2]_{\Sp[]}\bs{T_1^U<T_2^U}}  - \P[\theta_1,\theta_2]_{\Sp[]}\bs{T_1^U<T_2^U}\right).
\end{align*}
But
$\E[\theta_1,\theta_2]_{\Sp[]}\br{\P[\Theta^1_t,\theta_2]_{\Sp[]}\br\cdot}$
is the probability measure associated with the process
$\br{\Theta^1_{s+t}, \Theta^2_s}$, $s>0$, obtained by giving
$\Theta_1$ a headstart. So we have
\begin{align*}
\E[\theta_1,\theta_2]_{\Sp[]}\br{\P[\Theta^1_t,\theta_2]_{\Sp[]}\bs{T^U_1<T^U_2}}
&\geq \P[\theta_1,\theta_2]_{\Sp[]}\bs{T^U_1-t <T^U_2}
- \P[\theta_1,\theta_2]_{\Sp[]}\bs{T^U_1<t}\\
&\geq  \P[\theta_1,\theta_2]_{\Sp[]}\bs{T^U_1<T^U_2} -
\P[\theta_1,\theta_2]_{\Sp[]}\bs{T^U_1<t}
\end{align*}
and, since $\frac{1}{t}\P[\theta_1,\theta_2]_{\Sp[]}\br{T^U_1<t} \to
0$ as $t\to 0$, we may pass to the limit, and we see that $\calig
A^1h^U\br{\theta^1,\theta^2} \geq 0$.
\end{proof}

We are ready to prove Lemma \ref{sam:keylemma}.

\begin{proof}[Proof of Lemma \ref{sam:keylemma}]
Set $C = \setof{\lambda u}{\lambda\in\R^+,\, u\in U}$ for some open
subset $U\subset\Sp$ and let $\Y_t$ be a Fleming-Viot process in $C$.

We deal first  with the special case when $d=1$, in which case either
$ C=\R$ and there is nothing to prove or $C=\R^+$.  If $C=\R^+$ then
$\Y_t$ is a 2-dimensional Brownian motion in the quarter plane with
jumps $\br{y,0}\mapsto\br{y,y}$ or $\br{0,y}\mapsto\br{y,y}$ whenever
the process exits the first quadrant. As these jumps only increase
$\bmm{\Y_t}$ then $\bmm{\Y_t}$ dominates a $\bes\br 2$ process and
$\Y_t$ does not converge to 0.

For $d\geq 2$ define a function
\begin{equation*}
\mu\br x =
\begin{cases} \log{\bmm x}, &\text{if $d=2$,}\\
\frac{\bmm x^{2-d}}{2-d} &\text{if $d\geq 3$,}
\end{cases}
\end{equation*}
and define processes
\begin{align*}
\Phi^1_t &= \phi\br{Y^1_t}, & M^1_t &= \mu\br{Y^1_t},\\
\Phi^2_t &= \phi\br{Y^2_t},  & M^2_t &= \mu\br{Y^2_t},\\
H_t &= h^U\br{\Phi^1_t,\Phi^2_t}, & S_t &= M^1_t + \br{M^2_t -
M^1_t}H_t.
\end{align*}

Now, $\mu$ is harmonic on $\R^d$ and it will be key to our argument
that $M^1_t$ and $M^2_t$ are both local martingales except when
$\Y_t$ jumps.  We say a $\Y_t$-adatapted process $R_t$ is a
\emph{martingale between jumps} if $R_t -
\sum_{\setof{i\in\N}{\tau_i\leq t}}\br{R_{\tau_i} - R_{\tau_i^-}}$ is
a continuous local martingale. The process $S_t$ is a~convex
combination of $M^1_t$ and $M^2_t$, so if both $Y^1_t$ and $Y^2_t$
converge to the origin, then $S_t$ converges to $-\infty$. Notice
also that if $Y^1$ approaches $\partial C$ then $H_t\to 1$ and so
$S_t\to M^2_t$. Similarly, if $Y^2_t$ approaches the boundary then
$S_t\to M^1_t$. So $S_t$ is continuous.

Set
\begin{equation*}
    N_s = H_s - \int_0^s
    \left(\frac{\calig A^1h^U\br{\Phi_t}}{\bmm{B^1_t}^2}
    + \frac{\calig A^2h^U\br{\Phi_t}}{\bmm{B^2_t}^2}\right)dt.
\end{equation*}
By (\ref{sam:mgale}) $N_t$ is a martingale between jumps. We
may check that the cross variation terms $\ba{M^1,\Phi^1}_t =
\ba{M^2_t,\Phi^2_t} = 0$ and so, as $H_t$ is a $\calig C^2$
function of $\Phi^1_t$ and $\Phi^2_t$, we have $\ba{M^1,H}_t =
\ba{M^2_t,H}_t = 0$ and for $s\in \halfint{\tau_i, \tau_{i+1}}$
we may calculate
\begin{align*}
S_s &= S_{\tau_i} + \int_{\tau_i}^s {\br{1-H_t}}\,dM^1_t
+ \int_{\tau_i}^s {H_t}\,dM^2_t
+ \int_{\tau_i}^s \br{M^2_t - M^1_t} \,dH_t \\
&= S_{\tau_i} + \int_{\tau_i}^s {\br{1-H_t}}\,dM^1_t +
\int_{\tau_i}^s {H_t}\,dM^2_t
+ \int_{\tau_i}^s \br{M^2_t - M^1_t}
\,dN_t \\&\hphantom{= \mu\br{\xi_i}} +
\int_{\tau_i}^s \br{M^2_t-M^1_t}\br{\frac{\calig
A^1h^U\br{\Phi_t}}{\bmm{B^1_t}^2} + \frac{\calig
A^2h^U\br{\Phi_t}}{\bmm{B^2_t}^2}}dt.
\end{align*}
Therefore $S_s - \int_{\tau_i}^s
\br{M^2_t-M^1_t}\br{\frac{\calig
A^1h^U\br{\Phi_t}}{\bmm{B^1_t}^2} + \frac{\calig
A^2h^U\br{\Phi_t}}{\bmm{B^2_t}^2}}dt$ is a martingale between
jumps.

Now from Lemma \ref{sam:harmonic} we have $\calig A^1h^U = -\calig
A^2h^U \geq 0$ and so
\begin{equation*}
  \frac{\calig A^1h^U\br{\Phi_t}}
  {\bmm{B^1_t}^2} + \frac{\calig A^2h^U\br{\Phi_t}}
  {\bmm{B^2_t}^2} = \calig A^1h^U\br{\Phi_t}
  \br{{\bmm{B^1_t}^{-2}} - {\bmm{B^2_t}^{-2}} }.
\end{equation*}
But $\mu$ is an increasing function of the norm $\|\cdot\|$, so for
$\tau_i \leq s_1 \leq s_2 < \tau_{i+1}$,
\[\int_{s_1}^{s_2}
 \br{M^2_t-M^1_t}\br{\frac{\calig
A^1h^U\br{\Phi_t}}{\bmm{B^1_t}^2} + \frac{\calig
A^2h^U\br{\Phi_t}}{\bmm{B^2_t}^2}}dt \geq 0.\]
 Therefore $S_t$ is a continuous local submartingale and it cannot
converge to $-\infty$. Thus $\Y_t$ does not converge to $\br{0,0}$.
\end{proof}

\begin{corollary}\label{sam:cornoconverge}
If $\X_t$ is a Fleming-Viot process in a polyhedral domain $D$
then with probability one the sequence of jump points $\xi_i$ does
not converge to any $\xi_\infty\in \partial D$ as $i\to\infty$.
\end{corollary}

\begin{proof}
First, for $\sigma\in\partial\calig K$, let $F^\sigma$ be the event
that $\xi_i\to \xi_\infty$ for some $\xi_\infty\in\interior\sigma$
and assume without loss of generality that $\underline 0\in\sigma$.
Set
\begin{equation*}
    F_i^\sigma = F^\sigma\cap \bs{X^j_t\in\hood\sigma;\,
    t\geq \tau_i,\,j=1,2}.
\end{equation*}
Then, as $\hood\sigma$ is open in $\overline D$, from Lemma
\ref{sam:besselness}, $F^\sigma_i$ increases to $F^\sigma$ up
to an event of probability~$0$. By the strong Markov property
and Lemma \ref{sam:keylemma},
\begin{equation*}
    \P\br{F^\sigma_i} = \P[\xi_i]_\sigma\Bigl(\Y_t\to
    \br{\underline 0, \underline 0} \cap \bs{X^j_t\in\hood\sigma;\,
    t\geq \tau_i,\,j=1,2}\Bigr) = 0.
\end{equation*}
So as $\partial\calig K$ is a finite set of simplices we have
$\P\bs{\exists \xi_\infty\in\partial D \text{ s.t. }\xi_i\to
\xi_\infty \text { as } i\to\infty} = 0$.
\end{proof}

To complete the proof of Theorem \ref{sam:thm:main} we  consider
the set
\begin{equation*}
    L = \setof{\sigma\in\calig K}{\text{there exists a subsequence $\xi_{i_n}\to \xi\in\interior\sigma$ as $n\to\infty$}}.
\end{equation*}
It is easy to check that the event $\{\sigma\in L\}$ is
$\X$-measurable. We say $\sigma$ is a \emph{local maximum} of $L$ if
$L\cap\Star\sigma = \bc\sigma$.  Of course any non-empty subset of a
finite lattice contains at least one local maximum, and $L$ is non
empty by compactness of $\overline D$.  We will prove Theorem
\ref{sam:thm:main} by showing that for each $\sigma\in\calig K$ the
event that $\tau_\infty<\infty$ and $\sigma$ is a local maximum of
$L$ has probability $0$.

\begin{proof}[Proof of Theorem \ref{sam:thm:main}]
Fix $\sigma\in\partial \calig K$, and note that
$\hood\sigma\backslash\sigma$ is non empty. We show first that if
$\xi_i$ has a limit point in $\interior\sigma$ and
$\tau_\infty<\infty$, then $\xi_i$ has a second limit point in
$\hood\sigma\backslash\sigma$.

First suppose  that $\sigma = \bc{v}$ is a vertex of $\calig K$ and
$v$ is a limit point of $\xi_i$. By
Corollary~\ref{sam:cornoconverge}, the sequence $\xi_i$ does not
converge to $v$ as $i\to\infty$, so we may choose $\varepsilon>0$
such that $B\br{v,\varepsilon}\cap\overline D\subset \hood\sigma$ and
that $\bmm{\xi_i-v}>\varepsilon$ infinitely often.  If this is the
case then there are infinitely many pairs $(\xi_{i_n},\xi_{i_{n+1}})$
such that $\xi_{i_n}\in B\br{v,\varepsilon}$ and $\xi_{i_{n+1}}
\notin B\br{v,\varepsilon}$. But from Lemma \ref{sam:besselness} we
have $\bmm{\xi_i - \xi_{i+1}}\to 0$ as $i\to\infty$ hence
$\bmm{\xi_{i_n}-v}\to\varepsilon$ as $i\to\infty$. Therefore, as
$\partial B\br{v,\varepsilon}$ is compact, $\xi_i$ must have some
limit point in $\partial B\br{v,\varepsilon}\cap\overline D\subset
\hood\sigma\backslash\bc{v}$.

If $\sigma$ is a $k$-simplex for $0<k<d$ then for each $x\in\interior\sigma$ there exists $\varepsilon>0$ such that $B\br{x,2\varepsilon}\cap\overline D\subset \hood\sigma$. We will consider upcrossings of the interval $\bs{\varepsilon,2\varepsilon}$ by $\bmm{\xi_i-x}$. Define sequences $i_n, j_n\in\N\cup\bc\infty$ and
$T_n,\eta_n\in\R\cup\bc\infty$ by: $j_0 = 0,$
\begin{align*}
i_{n+1} &= \inf\setof{i>j_n}{\xi_{i}\in B\br{x,\varepsilon}}, \\
j_n &=  \inf\setof{j>i_n}{\xi_{j}\notin B\br{x,2\varepsilon}}, \\
T_n &= \inf\setof{t>\tau_{i_n}}{X^1_t\notin B\br{x,2\varepsilon} \text{ or } X^2_t\notin B\br{x,2\varepsilon} }, \\
\eta_n &= \sup_i\setof{\tau_i}{\tau_i<T_n}.
\end{align*}
Then we put $N = \sup\setof{n\in\N}{j_n<\infty}$ to be the
number of upcrossings.

Note that $B\br{x,2\varepsilon}\cap\overline D\subset \hood\sigma$
and so $\X_{\br{t +\tau_{i_n}}\wedge T_n}$ is a Fleming-Viot process
in $\Wedge\br\sigma$ started at $\br{\xi_{i_n},\xi_{i_n}}$ and
stopped on exiting $B\br{x,2\varepsilon}$. So we may consider
$\P[\xi_{{i_n}}]_\sigma$ and factorize $\X_t= \Y_t+\Z_t$ as in Lemma
\ref{sam:wedgelemma}. For $t\in\bs{\tau_{i_n},\eta_n}$, the process
$\tilde Z_t$ is measurable with respect to $
\X\bigm|_{\bs{\tau_{i_n},T_n}}$ which is distributed according to
$\P[\xi_{i_n}]_\sigma$. Hence $\tilde
Z\bigm|_{\bs{\tau_{i_n},\eta_n}}$ is a Brownian motion in
$\Span\sigma$ with respect to its own natural filtration.

Recall $\Z_{\tau_i} = \br{\zeta_i, \zeta_i}$ and set
\begin{equation*}
    \tilde V_t =
    \begin{cases}
      \bigl\|\tilde Z_t-\zeta_{i_n}\bigr\|, &\text{if  $t\in \bs{\tau_{i_n},\eta_n}$,}\\
      0, &\text{otherwise.}
    \end{cases}
\end{equation*}
Then $\tilde V_t$ is dominated by a $\bes\br{d}$ process reset to
zero at times $\tau_{i_n}$. So arguing as in the proof of Lemma
\ref{sam:besselness}, if $\tau_\infty<\infty$ and the number of
upcrossings $N=\infty$, then $\tau_{i_n} < \tau_\infty<\infty$  for
each $n\in\N$, and $\tilde V_t\to 0$ as $\tau_i\to\infty$. But
$\eta_n=\sup_i\setof{\tau_i}{\tau_i<T_n}$, hence $\X_{\eta_n} =
\br{\xi_{k_n},\xi_{k_n}}$ for some $k_n\in \N$ and either
$\bmm{x-X^1_{T_n}}=2\varepsilon$ or $\bmm{x-X^2_{T_n}}=2\varepsilon$.
So if $\tau_\infty<\infty$ and $N=\infty$, we must have
$\bmm{\xi_{k_n} - x}\to 2\varepsilon$ as $n\to\infty$ and so
$\xi_{k_n}$ has a limit point $\xi_\infty \in \partial
B\br{x,2\varepsilon}\cap\overline D \subset\hood\sigma$. But $\tilde
V_t\to 0$ as $t\to\tau_\infty$ with probability one,  so we cannot
have $\xi_\infty\in \sigma$ and we must have $\xi_\infty\in
\hood\sigma\backslash\sigma$.

Now let $Q^\sigma$ be a countably dense subset  of $\interior\sigma$
and suppose $\xi_i$ has some limit point $x\in\interior\sigma$. By
Corollary \ref{sam:cornoconverge},  $\xi_i$ does not converge to $x$
as $i\to\infty$ and we may choose some rational $\varepsilon>0$ such
that $B\br{x,3\varepsilon}\cap\overline D\subset \hood\sigma$ and
$\bmm{\xi_i-x}>3\varepsilon$ infinitely often. Now choose $q\in
Q^\sigma\cap B\br{x,\varepsilon}$ and notice that $\bmm{\xi_i-q}$
makes infinitely many upcrossings of the interval
$\bs{\varepsilon,2\varepsilon}$. If $\tau_\infty<\infty$ then as
$Q^\sigma$ is countable, with probability one we may find some limit
point $\xi_\infty\in \hood\sigma\backslash\sigma$.

Recall the definition of the set $L$.  As $L$ is nonempty there must
exist some local maximum~$\sigma$. However if $\tau_\infty<\infty$ then,
by Corollary \ref{sam:corrlobus}, we have $L\subseteq\partial \calig
K$. We have just shown that if $\tau_\infty<\infty$ then $L$ has no
local maximum in $\partial \calig K$. Hence we must have $\tau_\infty
= \infty$.
\end{proof}

\renewcommand{\Z}[1]{Z^{#1}}


\bibliographystyle{amsplain}
\bibliography{fv}

\end{document}